\documentclass{amsart}
\usepackage{latexsym,amsxtra,amscd,ifthen,amsmath,color, multicol}
\usepackage{amsfonts}
\usepackage{verbatim}
\usepackage{amsmath}
\usepackage{amsthm}
\usepackage{amssymb}
\usepackage[notcite,notref,final]{showkeys}
 \usepackage[all,cmtip]{xy}
 \usepackage[normalem]{ulem}

\usepackage{hyperref}

\numberwithin{equation}{section}

\theoremstyle{plain}
\newtheorem{theorem}{Theorem}[section]
\newtheorem{lemma}[theorem]{Lemma}

\newtheorem{proposition}[theorem]{Proposition}

\newtheorem{corollary}[theorem]{Corollary}

\theoremstyle{definition}
\newtheorem{definition}[theorem]{Definition}

\newtheorem{notation}[theorem]{Notation}
\newtheorem{routine}[theorem]{Routine}

\newtheorem{claim}[theorem]{Claim}
\newtheorem{remark}[theorem]{Remark}
\newtheorem{question}[theorem]{Question}

\makeatletter              
\let\c@equation\c@theorem  
\makeatother

\newcommand{\mb}{\mathbb}

\newcommand{\NN}{\mb N}

\newcommand{\PP}{\mb P}

\newcommand{\ZZ}{\mb Z}

\newcommand{\kk}{{\Bbbk}}

\newcommand{\mf}{\mathfrak}

\newcommand{\ssm}{\smallsetminus}
\newcommand{\wt}{\widetilde}
\newcommand{\ang}[1]{\langle #1 \rangle}

\newcommand{\wh}{\widehat}

\newcommand{\DMO}{\DeclareMathOperator}
\DMO{\Aut}{Aut}
\DMO{\Div}{div}
\DMO{\ev}{ev}
\DMO{\Bs}{Bs}
\DMO{\GKdim}{GKdim}
\DMO{\Pic}{Pic}
\DMO{\Num}{Num}
\DMO{\Cl}{Cl}
\DMO{\CaCl}{CaCl}
\DMO{\im}{im}
\DMO{\hilb}{hilb}

\newcommand{\lp}{\textnormal{(}}
\newcommand{\rp}{\textnormal{)}}

\providecommand{\abs}[1]{\lvert#1\rvert}

\newcommand{\beq}{\begin{equation}}
\newcommand{\eeq}{\end{equation}}

\DMO{\ad}{ad}

\begin{document}

\title[Maps from the enveloping algebra of the positive Witt algebra]
{Maps from the enveloping algebra of the positive Witt algebra to regular algebras}

\author{Susan J. Sierra and Chelsea Walton}

\address{Sierra: School of Mathematics, The University of Edinburgh, Edinburgh EH9 3FD, United Kingdom}

\email{s.sierra@ed.ac.uk}

\address{Walton: Department of Mathematics, Temple University, Philadelphia, Pennsylvania 19122,
USA}

\email{notlaw@temple.edu}

\bibliographystyle{alpha}       

\begin{abstract} We construct homomorphisms from the universal enveloping algebra of the positive (part of the) Witt algebra to several different Artin-Schelter regular algebras, and determine their kernels and images. As a result, we produce elementary proofs that the universal enveloping algebras of the Virasoro algebra, the Witt algebra, and the positive Witt algebra are neither left nor right noetherian. 
\end{abstract}



\maketitle


\setcounter{section}{-1}


\section{Introduction}\label{INTRO2}

Let $\kk$ be a  field of characteristic 0.
All vector spaces,  algebras, $\otimes$ are over $\kk$, unless stated otherwise. In this work, we construct and study homomorphisms from the universal enveloping algebra of the positive part of the Witt algebra to {\it Artin-Schelter} ({\it AS-}){\it regular algebras}. The latter serve as homological analogues of commutative polynomial rings in the field of noncommutative algebraic geometry. 

To begin,  consider the Lie algebras below.

\begin{definition}[$V$, $W$, $W_+$]  \label{idef:Liealgs} 
We define the following Lie algebras:  
\begin{enumerate}
\item The {\it Virasoro algebra} is defined to be the Lie algebra $V$ with basis $\{e_n\}_{n \in \mathbb{Z}} \cup \{c\}$ and Lie bracket $[e_n,c]= 0$, $[e_n, e_m] = (m-n)e_{n+m} + \frac{c}{12}(m^3 - m) \delta_{n+m,0}$. 
\smallskip

\item The {\it Witt} (or {\it centerless Virasoro}) {\it algebra} is defined to be the Lie algebra $W$ with basis $\{e_n\}_{n \in \mathbb{Z}}$ and Lie bracket $[e_n, e_m] = (m-n)e_{n+m}$. 
\smallskip

\item The {\em  positive \textnormal{(}part of the\textnormal{)} Witt algebra}  is defined to be the Lie subalgebra $W_+$ of $W$ generated by $\{e_n\}_{n \geq 1}$.
\end{enumerate}
For any Lie algebra $\mf{g}$, we denote its universal enveloping algebra by $U(\mf g)$.
\end{definition}

Further, consider the following algebras.

\begin{notation}[$S$, $R$] \label{inot0} 
Let $S$ be the algebra generated by $u,v,w$, subject to the relations
$$uv-vu-v^2 ~=~ uw-wu-vw ~=~vw-wv ~=~ 0.$$
Let $R$ be the {\it Jordan plane}, generated by $u,v$, subject to the relation
$uv-vu-v^2= 0.$
\end{notation}
It is well-known that  $R$ is an AS-regular algebra of global dimension 2. Moreover, we see by Lemma~\ref{Zhang} that $S$ is also AS-regular, of global dimension 3. 

This work focuses on  maps that we construct from  the enveloping algebra $U(W_+)$ to both $R$ and $S$, given as follows:  

\begin{definition}[$\phi$, $\lambda_a$] \label{idef:maps}
Let
$ \phi: U(W_+) \to S $ be the algebra homomorphism induced by defining 
\beq\label{phi} \phi(e_n) = (u-(n-1)w)v^{n-1}.\eeq
For $a\in \kk$, let $\lambda_a: U(W_+) \to R$ be the algebra homomorphism induced by defining \beq \label{lambda} \lambda_a(e_n) = (u-(n-1)av)v^{n-1}.\eeq

\noindent That $\phi$ and $\lambda_a$ are well-defined is Lemma~\ref{lem:homom}.\end{definition}

Our main result is that we understand fully the kernels and images of the maps above, as presented below.

\begin{theorem} \label{ithm2} We have the following statements about the kernels and images of the maps $\phi$ and $\lambda_a$. 
\begin{enumerate}
\item[(a)] \textnormal{[Propositions~\ref{prop:(d)}, \ref{prop:(d)fora}]} $\ker \lambda_a$ is equal to  the ideal $\lp e_1 e_3-e_2^2-e_4 \rp$ if $a =0,1$; or is an
 ideal generated by one element of degree 5 and two elements of degree 6 \lp listed in Proposition~\ref{prop:(d)fora}\rp ~if $a \neq 0,1$.
\item[(b)] \textnormal{[Proposition~\ref{prop:(c)}]} $\lambda_a(U(W_+))$ is equal to $\kk + u R$ 
if $a =0$; ~is equal to $\kk + Ru$ 
if $a =1$;~ or  contains $R_{\geq 4}$ if $a \neq 0,1$.  
For all $a$, the image of $\lambda_a$ is noetherian.
\item[(c)] \textnormal{[Theorem~\ref{thm:(b)}]} $ \ker \phi$ is equal to  $(e_1e_5 - 4e_2 e_4+ 3e_3^2+2e_6)$.
\end{enumerate}
\end{theorem}

\noindent The image of  $\phi$ will be discussed later in the introduction, after Theorem~\ref{ithm4}.

The result above has  a surprising application.   
In \cite[Theorem~0.5 and Corollary~0.6]{SieWal:Witt}, the authors established that  $U(W_+)$, $U(W)$, $U(V)$ are neither left nor right noetherian through relatively indirect means, using the techniques  of \cite{S-surfclass}.  
In particular, we were not able to give an example of a non-finitely-generated right or left ideal in any of these enveloping algebras. 
However, in the course of proving Theorem~\ref{ithm2}, we produce an elementary and constructive proof of \cite[Theorem~0.5 and Corollary~0.6]{SieWal:Witt}. Namely, we obtain: 

\begin{theorem}[Proposition~\ref{prop:(d)}, Theorem~\ref{thm:easy}] \label{ithm1}
The ideal $\ker \lambda_0= \ker \lambda_1= ( e_1 e_3-e_2^2-e_4 )$   is not finitely generated as either a left or a right ideal of $U(W_+)$. 
\end{theorem}

We prove this theorem  by noting that $\lambda_0$ factors through $\phi$, and by studying $B: = \phi(U(W_+))$.   A key step  is to   compute $I := \phi(\ker \lambda_0)$, and to show that $I$ is not finitely generated as a left or right ideal of $B$.  

Note that the map \eqref{lambda} can be extended to $W$ to define a map which we denote by $$\widehat{\lambda}_a:  U(W) \to R[v^{-1}].$$  
We also have: 

\begin{theorem}[\eqref{fact}, Theorem~\ref{thm:easy2}] \label{ithm:easy2}
The ideal $\ker \widehat{\lambda}_0 = \ker \widehat{\lambda}_1$ is not finitely generated as either a left or right ideal of $U(W)$.
\end{theorem}

We remark that $R[v^{-1}]$ is isomorphic to the ring $\kk[x,x^{-1}, \partial]$, which is a familiar  localization of the Weyl algebra. To see this,  set $v = x$ and $u  = x^2 \partial$, so $\partial x = x\partial +1$.  Then, $uv-vu = x^2 = v^2$.  
We obtain that
 \[\widehat{\lambda}_1(e_n) = v^{n-1} u = x^{n+1} \partial.\]
Thus, $\widehat{\lambda}_1$ is a well-known homomorphism.  

We now compare  Theorem~\ref{ithm1} with our earlier proof (in \cite{SieWal:Witt}) that $U(W_+)$ is not left or right noetherian.
The earlier proof used a ring homomorphism $\rho$ with a more complicated definition: 

\begin{notation}[$X$, $f$, $\tau$, $\rho$] \label{inot1} 
Take $\PP^3:=\PP^3_\kk$ with coordinates $w,x,y,z$.  
 Let $X= V(xz-y^2) \subseteq \PP^3$ be the projective cone over a smooth conic in $\PP^2$. 
 
 Define an automorphism $\tau$ of $ X$ by
$$\tau([w:x:y:z]) = [w-2x+2z:z:-y-2z:x+4y+4z].$$
 Denote the pullback of  $\tau$ on $\kk(X)$ by $\tau^*$, so that $g^\tau := \tau^* g = g \circ \tau$ for $g \in \kk(X)$.  
 Form the ring $\kk(X)[t; \tau^*]$ with multiplication $tg = g^\tau t$ for all $g \in \kk(X)$. 
Let  $$f:= \displaystyle \frac{w+12x+22y+8z}{12x+6y},$$ 
considered as a rational function in $\kk(X)$.
Now let 
$\rho: U(W_+) \to \kk(X)[t; \tau^*]$ be the algebra homomorphism induced by  setting $\rho(e_1) = t $ and $\rho(e_2) = ft^2$.
 \end{notation}
 \noindent
That $\rho$ is well-defined is \cite[Proposition~1.5]{SieWal:Witt}. 

The method \cite{SieWal:Witt} made a number of reductions to show that $\rho(U(W_+))$ is not left or right noetherian. 
That proof
can now be streamlined via the next result.

\begin{theorem}[Theorem~\ref{thm:kernels}] \label{ithm4}
We have that $\ker \rho = \ker \phi = \bigcap_{a \in \kk} \ker \lambda_a$.  
\end{theorem}

Since we show that $\phi(U(W_+))$ is not left or right noetherian in the course of proving Theorem~\ref{ithm1}, we have 
  by Theorems~\ref{ithm2}(c) and~\ref{ithm4} that   $\rho(U(W_+)) \cong \phi(U(W_+)) \cong U(W_+)/(e_1e_5 - 4e_2 e_4+ 3e_3^2+2e_6)$ is neither left or right noetherian.

We end by discussing an open question that was first brought to our attention by Lance Small. 

\begin{question} \label{ques:ACC} Does $U(W_+)$ satisfy the ascending chain condition on {\em two-sided} ideals?
\end{question}
  
Our result here is only partial:  we show that

\begin{proposition}[Proposition~\ref{prop3}] \label{iprop3}
The ring $B := \phi(U(W_+))$ satisfies the ascending chain condition on two-sided ideals.
\end{proposition}
\noindent Of course, this yields no direct information on the question for $U(W_+)$.

We have the following conventions throughout the paper. We take
$\NN = \ZZ_{\geq 0}$ to be the set of non-negative integers. 
If $r$ is an element of a ring $A$, then $(r)$ denotes the two-sided ideal $ArA$ generated by $r$.
If $A = \bigoplus_{n \in \ZZ} A_n$ is a graded $\kk$-algebra (or graded module), then we define the Hilbert series
\[ \hilb A = \textstyle \sum_{n \in \ZZ} \dim_\kk A_n t^n.\]

This article is organized as follows. We present preliminary results in Section~\ref{sec:prelim}, including an alternative way of multiplying elements in $S$ and in $R$ (Lemma~\ref{Zhang}); this method will be employed throughout, sometimes without mention. In Section~\ref{LAMBDA}, we discuss the maps $\lambda_a$ and prove parts (a,b) of Theorem~\ref{ithm2}. 
In Section~\ref{sec:easy} we use the map $\lambda_0$ to establish Theorem~\ref{ithm1}; we also prove Theorem~\ref{ithm:easy2}. 

Before proceeding to study the map $\phi$, we present its connection with the map $\rho$, the key homomorphism of \cite{SieWal:Witt}. Namely, in Section~\ref{sec:kernels}, we establish Theorem~\ref{ithm4}. Then in Section~\ref{PHI}, we verify part (c) of Theorem~\ref{ithm2}. Our last result, Proposition~\ref{iprop3}, is presented in Section~\ref{ACC}.  Proofs of computational claims via Maple and Macaulay2 routines and a known result in ring theory to which we could not find a reference are provided in the appendix.


\section{Preliminaries} \label{sec:prelim}

The main focus of this paper is the universal enveloping algebra of the positive Witt algebra, $W_+$.
To begin, we recall some basic facts about the algebra $U(W_+)$.

\begin{lemma} \label{lem:present}
Recall Definition~\ref{idef:Liealgs}\lp c\rp.
\begin{enumerate} \item
 We have the following isomorphism:
$$U(W_+) \cong
\frac{\kk \langle e_1, e_2 \rangle}
{\left( 
\begin{array}{c}
[e_1,[e_1,[e_1,e_2]]]+6[e_2,[e_2,e_1]],\\
\left[e_1,\left[e_1,\left[e_1,\left[e_1,\left[e_1,e_2\right]\right]\right]\right]\right]
+40[e_2,[e_2,[e_2,e_1]]] 
\end{array} 
\right)}.$$ 
\vspace{.0in}
\item
The set
$ \{ e_{i_1} e_{i_2} \dots e_{i_k} \ | \ k \in \NN \mbox{ and } 1 \leq i_1 \leq i_2 \leq \dots \leq i_k \in \NN \}$
forms a $\kk$-vector space basis of~$U(W_+)$.
\end{enumerate}
\end{lemma}

\begin{proof}
Part (a) is
 \cite[Lemma~1.1]{SieWal:Witt}, 
 and part (b) is clear from the definition of $U(W_+)$.  
  \end{proof}

Next, 
let us present some notation that we will use for the rest of the paper.  We will work with the algebras $R$ and $S$ defined in Notation~\ref{inot0}; note that we can view $R$ as a subalgebra of $S$.    In addition:

\begin{notation}[$Q$]\label{not:Q}
Take  $Q$ to be the  subalgebra of $S$ generated by $u$, $v$, and $vw$.  
\end{notation}

 In our first result, we provide an easy way to multiply elements in $S$. Recall from \cite{Zhangtwist} that a {\it Zhang twist} of  a graded algebra $L$, by an automorphism $\mu$ of $L$, is the algebra $L^\mu$,
 where $L^\mu = L$ as graded vector spaces and $L^\mu$ has multiplication $\ell \ast \ell' = \ell (\ell')^{\mu^{i}}$ for $\ell \in L_i$ and $\ell' \in L$.
 
 Moreover, recall that an {\it Artin-Schelter} ({\it AS-}){\it regular algebra} is a connected graded algebra $A$ of finite global dimension, of finite injective dimension $d$ with $\text{Ext}_A^i({}_A \kk, {}_A A) \cong \text{Ext}_A^i(\kk_A, A_A) \cong \delta_{i,d} \kk$ (that is, $A$ is {\it AS-Gorenstein}), and has finite Gelfand-Kirillov dimension.
 These algebras are important in noncommutative ring theory because they are noncommutative analogues of polynomial rings and share many of their good properties.

\begin{lemma}[$\mu$, $\nu$] \label{Zhang}
Let $\mu \in \Aut(\kk[x,y,z])$ be defined by  
$$\mu(x) = x-y, \quad \mu(y)=y, \quad \mu(z)=z.$$
Then,  $S$ is isomorphic to the Zhang twist  $\kk[x,y,z]^\mu$.  
Further, $\mu$ restricts to an  automorphism of $\kk[x,y,yz]$, which we also denote by $\mu$, and  to an automorphism $\nu$ of $\kk[x,y]$.  
We have that
 $R \cong \kk[x,y]^{\nu}$ 
 and 
 \linebreak $Q \cong \kk[x,y,yz]^\mu$
 as graded $\kk$-algebras. As a consequence, $S$, $R$, and $Q$ are AS-regular algebras.
 \end{lemma}

\begin{proof}
To see that $S \cong \kk[x,y,z]^\mu$, we emphasize that
\begin{equation} \label{multiplyS}
\begin{array}{c}
\text{the variables $u, v, w$ of $S$ have noncommutative multiplication,}\\ 
\text{the variables $x, y, z$ of $\kk[x,y,z]$ have commutative multiplication, and}\\ 
\text{$\ast$ denotes the noncommutative multiplication on $\kk[x,y,z]^\mu$ defined by}\\
\text{$\ell \ast \ell' = \ell (\ell')^{\mu^{i}}$ for $\ell \in \kk[x,y,z]_i$ and $\ell' \in \kk[x,y,z]$.}
\end{array}
\end{equation}
Now,
\[
\begin{array}{lllllll}
y \ast x &= yx^\mu &= y(x-y) &= (x-y) y &= xy - y^2 &= xy^\mu - yy^\mu &= x \ast y - y \ast y,\\
 z \ast x &= zx^\mu &= z(x-y) &= (x-y) z &= xz - yz &= xz^\mu - yz^\mu &= x \ast z - y \ast z ,\\
  z \ast y & = zy^\mu &= zy &= yz &= yz^\mu &&= y \ast z.
\end{array}
\]
Thus,  if we identify $u, v, w$ with $x, y, z$, respectively, then  the relations of $S$ hold in $\kk[x,y,z]^\mu$, and $S \cong \kk[x,y,z]^\mu$ as graded $\kk$-algebras. 

That $\mu$ restricts to automorphisms of $\kk[x,y]$ and $\kk[x,y,yz]$ is immediate, and the other isomorphisms hold by a similar argument. Moreover, the last statement follows as commutative polynomial rings are AS-regular and this property is preserved under Zhang twist  by \cite[Theorem~1.3(i)]{Zhangtwist}.
\end{proof}

Now we verify that the algebra homomorphisms $\lambda_a$ and $\phi$ from Definition~\ref{idef:maps} are well-defined.

\begin{lemma} \label{lem:homom}
The maps $\phi$ and $\lambda_a$ of Definition~\ref{idef:maps} are well-defined homomorphisms of graded $\kk$-algebras. 
\end{lemma}

\begin{proof}
We check that $\phi$ respects the Witt relations given in 
Definition~\ref{idef:Liealgs}(b),
by using Lemma~\ref{Zhang} and~\eqref{multiplyS}:
\[
{\small \begin{array}{rl}
\phi(e_ne_m - e_m e_n)
&= (u-(n-1)w)v^{n-1} (u-(m-1)w)v^{m-1} - (u-(m-1)w)v^{m-1}(u-(n-1)w) v^{n-1} \\
 &= (x-(n-1)z)(x-(m-1)z)^{\mu^{n}} y^{n+m-2} - (x-(m-1)z)(x-(n-1)z)^{\mu^m} y^{n+m-2}\\
 &= [(x-(n-1)z)(x-ny-(m-1)z)-(x-(m-1)z)(x-my-(n-1)z)] y^{n+m-2} \\
 &= (m-n)xy^{n+m-1} + (n(n-1)-m(m-1)) y^{n+m-1}z\\ 
 &= (m-n)(x-(n+m-1)z) y^{n+m-1}\\ 
 &= (m-n)(u-(n+m-1)w) v^{n+m-1}\\ 
 &= (m-n)\phi(e_{n+m}).
\end{array}}
\]
So, the claim holds for $\phi$.  

Similarly, we verify that $\lambda_a$ respects the Witt relations:
\[
{\small \begin{array}{rl}
\lambda_a(e_ne_m - e_m e_n)
&= (u-(n-1)av)v^{n-1} (u-(m-1)av)v^{m-1} - (u-(m-1)av)v^{m-1}(u-(n-1)av) v^{n-1} \\
&= [(x-(n-1)ay)(x-ny-(m-1)ay)-(x-(m-1)ay)(x-my-(n-1)ay)] y^{n+m-2} \\
 &= (m-n)(x-a(n+m-1)y) y^{n+m-1}\\ 
 &= (m-n)(u-a(n+m-1)v) v^{n+m-1}\\ 
 &= (m-n)\lambda_a(e_{n+m}).
\end{array}}
\]
Thus, the claim holds for $\lambda_a$.
 \end{proof}
 
 Next, we define the key algebras $A(a)$ and $B$ that we will use  throughout.
  
 \begin{notation}[$A(a)$, $B$] \label{not1} 
Take $a \in \kk$ and let $A(a)$  denote the subalgebra $\lambda_a(U(W_+))$ of $R$. Further, let $B$ denote the subalgebra $\phi(U(W_+))$ of $S$.
\end{notation}

We point out a useful observation. 

\begin{lemma} \label{lem:BinQ} 
We have that $B \subseteq Q$. 
\end{lemma}

\begin{proof}
We get that $\phi(e_1) = u$ and $\phi(e_2) = (u-w)v = uv-vw$ are in $Q$.  By Lemma~\ref{lem:present}(a), $B$ is generated by these elements, so we are done.
\end{proof}


\section{The kernel and image of the maps $\lambda_a$}\label{LAMBDA}

The goal of this section is to analyze the maps $\lambda_a$ from 
Definition~\ref{idef:maps}, which are well-defined by Lemma~\ref{lem:homom}.  In particular, we verify Theorem~\ref{ithm2}(a,b). 

To proceed, recall
Notations~\ref{inot0} and~\ref{not1}.
We first compute the factor rings $A(a)$, proving  Theorem~\ref{ithm2}(b).

\begin{proposition} \label{prop:(c)}
We have that $A(0) = \kk + uR$ \lp a right idealizer in $R$\rp, that $A(1) = \kk + Ru$ \lp a left idealizer in $R$\rp, and that $A(a)_{\geq 4} = R_{\geq 4}$  if $a \neq 0,1$. For all  $a$, the ring $A(a)$ is noetherian.
\end{proposition}

\begin{proof}
Recall from Lemma~\ref{lem:present}(a) that $U(W_+)$ is generated by $e_1$ and $e_2$. We have that $\lambda_0(e_1) = u$ and $\lambda_0(e_2) = uv$.
These elements generate $\kk + uR$. 
Moreover, $\lambda_1(e_1) = u$ and $\lambda_1(e_2) = (u-v)v = vu$, and these elements generate $\kk + Ru$. That the rings $A(0)$ and $A(1)$ are noetherian follows from \cite[Lemma~2.2(iii) and Theorem~2.3(i.a)]{SZ}.

When $a \neq 0,1$, we must show that $R_{\geq 4} \subseteq A(a)$. 
Since $u R_n + R_n u = R_{n+1}$ for $n \geq 1$ and since $\dim_{\kk}R_4=5$, the proof boils  down to showing that the set of elements
$$\lambda_a(e_1^4), \quad \lambda_a(e_1^2e_2), \quad \lambda_a(e_1e_2e_1), \quad \lambda_a(e_2e_1^2), \quad \lambda_a(e_2^2),$$
is $\kk$-linearly independent, for $a \neq 0,1$. Using Lemma~\ref{Zhang} and \eqref{multiplyS}, consider the following calculations,
\[
\begin{array}{lllll}
\lambda_a(e_1^4) &= u^4 &= xx^{\mu}x^{\mu^2}x^{\mu^3} &= x(x-y)(x-2y)(x-3y) &=:r_1,\\
\lambda_a(e_1^2e_2) &= u^2(u-av)v &= xx^{\mu}(x-ay)^{\mu^2}y^{\mu^3} &= x(x-y)(x-(2+a)y)y &=:r_2,\\
\lambda_a(e_1e_2e_1) &= u(u-av)vu &= x(x-ay)^{\mu}y^{\mu^2}x^{\mu^3} &= x(x-(1+a)y)y(x-3y) &=:r_3,\\
\lambda_a(e_2e_1^2) &= (u-av)vu^2 &= (x-ay)y^{\mu}x^{\mu^2}x^{\mu^3} &= (x-ay)y(x-2y)(x-3y) &=:r_4,\\
\lambda_a(e_2^2) &= (u-av)v(u-av)v &= (x-ay)y^{\mu}(x-ay)^{\mu^2}y^{\mu^3} &= (x-ay)y(x-(2+a)y)y&=:r_5. 
\end{array}
\]
By direct computation, we see that $r_1, \dots, r_5$ are linearly independent if $a \neq 0, 1$. 

Further, since $A(a)$ and $R$ are equal in large degree and $R$ is noetherian, $A(a)$ is noetherian by \cite[Lemma~1.4]{Stafford85}.  
\end{proof}

Our next goal is to compute the kernels of the maps $\lambda_a$ and establish Theorem~\ref{ithm2}(a).  
We will use the following notation:

\begin{notation}[$\pi$, $\pi_a$,  $\pi_B$] \label{not2} 
Let $\kk \langle t_1, t_2 \rangle$ be the free algebra, which we grade by setting $\deg t_i = i$.
 We set the notation below:
\begin{itemize}
\item $\pi: \kk \langle t_1,t_2 \rangle \to U(W_+)$ is   the algebra map given by $\pi(t_1) = e_1$ and $\pi(t_2) = e_2$;
\item $\pi_a: \kk \langle t_1,t_2 \rangle \to R$ is  the algebra map given by $\pi_{a}(t_1) = \lambda_a(e_1) = u$ and $\pi_{a}(t_2) = \lambda_a(e_2) = (u-av)v$, for $a \in \kk$.  The image of $\pi_{a}$ is $A(a)$.  Note that $\pi_{a} = \lambda_a \circ \pi$.
\item $\pi_B: \kk \ang{t_1,t_2} \to S$ is  the algebra map given by $\pi_B(t_1) =\phi(e_1)= u$ and $\pi_B(t_2) =\phi(e_2)= uv-vw$.  \linebreak The image of $\pi_B$ is $B$.  Note that $\pi_B = \phi \circ \pi$.
\end{itemize}
\end{notation}

In the next result, we   compute a presentation of  the algebra $A(0)$.

\begin{lemma} \label{towards(d)}
The kernel  of $\pi_0$ is generated by
\[
\begin{array}{l}
q:= t_1^2t_2-t_2t_1^2-2t_2^2,\\
q':= t_1^3t_2 - 3t_1^2t_2t_1+3t_1t_2t_1^2-t_2t_1^3+6t_2^2t_1-12t_2t_1t_2+6t_1t_2^2
\end{array}
\]
as a two-sided ideal.
\end{lemma}

\begin{proof} 
Let $A = A(0)$, and consider the exact sequence of right $A$-modules:
\[
\xymatrix{0 \ar[r] & K \ar[r]  & A[-1] \oplus A[-2] \ar[r]^(0.7){(u,uv)}   & A \ar[r]  & \kk \ar[r]  & 0.}
\] 
We make the following claim:
\medskip

\noindent {\bf Claim.} As a right $A$-module, $K$ is generated by $(u^2v, -u(u+2v))$ and $(u^2v^2, -u(u+2v)v)$.
\medskip

Assume the claim.   
It is well-known that one may deduce generators and relations of a connected graded $\kk$-algebra from the first few terms in a minimal resolution of the trivial module $\kk$.  
The precise method is given in Proposition~\ref{propappendix} in the appendix.  Using the notation of that result: take $b^1_1 = u^2v$, $b^1_2 = -u(u+2v)$, $b^2_1 = u^2v^2$, and $b^2_2 = -u(u+2v)v$. 
Moreover, take $\wt{b}^1_1 = t_1t_2$, $\wt{b}^1_2 = -t_1^2-2t_2$, $\wt{b}^2_1 = t_1^2t_2-t_1t_2t_1$, and $\wt{b}^2_2 = 2t_2t_1-3t_1t_2$.
Note that 
 $\pi_0(\wt{b}^i_j) = b^i_j$ for $i,j=1,2$. Now  we obtain by Proposition~\ref{propappendix} that
\[
\begin{array}{ll}
q_1:= t_1(\wt{b}^1_1) + t_2(\wt{b}^1_2) &= t_1^2t_2-t_2t_1^2-2t_2^2 \\
q_2:= t_1(\wt{b}^2_1) + t_2(\wt{b}^2_2) &= t_1^3t_2-t_1^2t_2t_1 + 2t_2^2t_1-3t_2t_1t_2
\end{array}
\]
generate $\ker \pi_0$. Observe that $q=q_1$ and that
$$q' - 4q_2 ~=~ -3t_1^3t_2+t_1^2t_2t_1 + 3t_1t_2t_1^2 -t_2t_1^3 -2t_2^2t_1 +6t_1t_2^2 ~=~ -3t_1q +qt_1 ~\in (q).$$
Thus, $\ker \pi_0$ is generated by $q$ and $q'$, as desired.

So it remains to prove the claim.
\medskip

\noindent {\it Proof of Claim.} Note that there is an isomorphism of graded right $A$-modules $\beta: uA \cap uvA \to K$ given by $\beta(r) = (u^{-1}r, -(uv)^{-1}r)$. 

Take $M:=A \cap vA$.  Since $A = \kk + uR$, it is easy to show that $M=uR \cap vuR$, and in particular, is a right $R$-module. 
Since $(uR + vu R)_{\geq 2} = R_{\geq 2}$, we get that $\dim_\kk M_n = \dim_\kk R_{n-1} + \dim_\kk R_{n-2} - \dim_\kk R_n = n-2$ for $n \geq 2$, and $\dim_\kk M_n = 0$ for $n< 2$.  
Moreover, $u^2 v = vu(u+2v) \in M$, so $u^2vR \subseteq M$ and $\hilb (u^2vR) = \hilb M$. So, $M = u^2v R$.
Now 
$$ uA \cap uvA = uM ~=~ u^3vR \overset{(*)}{=} u^3vA+u^3v^2A ~=~ uvu(u+2v)A + uvu(u+2v)vA,$$ 
where the equality $(*)$ holds as $R = A + vA$. Apply the map $\beta$ to the right-hand-side of the equation above to yield the desired result.
\end{proof}

We can now understand $\ker \lambda_0$ and $\ker \lambda_1$.  
We first prove that:

\begin{lemma}\label{lem1}
We have $\ker {\lambda}_0 = \ker {\lambda}_1$.
\end{lemma}
\begin{proof}
Working in the quotient division ring of $R$, we have:  $u^{-1}{ \lambda}_0(e_n) u = v^{n-1} u = {\lambda}_1(e_n)$.
So for any $f \in U(W_+)$, we have ${\lambda}_1(f) = u^{-1}{ \lambda}_0(f) u$.  The result follows. 
\end{proof}

\begin{proposition} \label{prop:(d)}
We have that $\ker \lambda_a = \lp e_1 e_3-e_2^2-e_4 \rp$ for $a =0,1$.
\end{proposition}

\begin{proof}
We first check that $e_1 e_3-e_2^2-e_4$ is indeed in $ \ker \lambda_0$ as follows:
\[
\lambda_0(e_1 e_3-e_2^2-e_4) = u(uv^2) - (uv)(uv) - uv^3                                                    =~ u^2v^2 - u(uv-v^2)v - uv^3 ~=~ 0.
\]

Recall that $\pi_0 = \lambda_0 \circ \pi$. 
So, Lemma~\ref{towards(d)} implies that $\ker \lambda_0 = \pi(\ker \pi_0)$ is generated by elements $\pi(q)$ and $\pi(q')$ in $U(W_+)$. Now $\pi(q') = 0$ by Lemma~\ref{lem:present}(a), so $\ker \lambda_0$ is generated by $\pi(q)$. 
Moreover, 
$$\pi(q) = e_1^2e_2-e_2e_1^2 -2e_2^2 = 2\left(e_1(e_1e_2-e_2e_1)-e_2^2-\left(\frac{1}{2}e_1^2e_2-e_1e_2e_1+\frac{1}{2}e_2e_1^2\right)\right)= 2(e_1 e_3-e_2^2-e_4),$$ 
using the relation $[e_n,e_m] = (m-n)e_{n+m}$ in $U(W_+)$. Thus, $\ker \lambda_0 = (e_1 e_3-e_2^2-e_4)$, as claimed. 

The result for $a=1$ now follows by Lemma~\ref{lem1}.
\end{proof}

It remains to analyze $\ker \lambda_a$ with $a \neq 0,1$. We do this in the next two results.

\begin{lemma} \label{lem:lambdaa}
For $a \neq 0,1$,  the kernel of $\lambda_a$ is generated in degrees 5 and 6. 

\end{lemma}

\begin{proof}
Take $A':=A(a)$. It suffices to show that the kernel of $\pi_{a}$ is generated in degrees 5 and 6.
Consider the exact sequence of right $A'$-modules:
$$\xymatrix{0 \ar[r] & K \ar[r] & A'[-1] \oplus A'[-2]  \ar[rr]^(0.65){(u,(u-av)v)}&& A' \ar[r] & \kk \ar[r] & 0.}$$ 
We have that $uA' \cap (u-av)vA' \cong K$ as right $A'$-modules.
As in the proof of Lemma~\ref{towards(d)}, it now suffices to show that $uA' \cap (u-av)vA'$ is generated in degrees 5 and 6 as a right $A'$-module.

Let $J:=uA' \cap (u-av)vA'$, and let $L := uR \cap (u-av)vR$. 
Note that $J \subseteq L$. Since $a \neq 0$, we get that $R_{\geq 2} = (uR+(u-av)vR)_{\geq 2}$. 
So, $\dim_\kk L_n = \dim_\kk R_{n-1} + \dim_\kk R_{n-2} - \dim_\kk R_n = n-2$, for $n \geq 2$. So, $\dim_\kk L_3 =1$, and principally generated as a right $R$-module by an element of degree 3. In fact,  
\begin{equation} \label{L=rR}
L = rR, \quad \text{ where } r:=u(uv+(1-a)v^2) = (uv-av^2)(u+2v).
\end{equation}
Since $A'_{\geq 4} = R_{\geq 4}$ by Proposition~\ref{prop:(c)}, we have $J_{\geq 6} = L_{\geq 6}$.  
By direct computation, one obtains that $J_i =0$ for $i=0, \dots, 4$; one can also use Routine~\ref{rout:kernel} in the appendix. 

Let $J' = J_5 A' + J_6 A'$.  
We prove by induction that
$J_n = J'_n$,
for all $n \geq 5$. 
The statement is clear for $n=5,6$. For $n=7$, we make the following assertion, the proof of which is presented in the appendix; see Claim~\ref{aclaim1}.

\medskip

\noindent {\bf Claim.} We have that  $J_5 A'_2 \not \subseteq J_6 A'_1$.

\medskip

\noindent 
So for $n \geq 6$, we have $J_n = L_n = r R_{n-3}$.  So  $\dim_{\kk} J_7 = 5$, and $\dim_{\kk} J_6 A'_1 = \dim_\kk J_6 = 4$.  With the claim, we obtain that $J_7 = J_5 A'_2 + J_6 A'_1$. Thus, $J_7 = J'_7$.
Now for the induction step, suppose we have established that $J'_n = J_n$ and $J'_{n-1} = J_{n-1}$ for some $n \geq 7$.  
Then,
\[
\begin{array}{rlll}
J_{n+1}  \supseteq J'_{n+1} &= J'_n u + J'_{n-1} (u-av)v  &= J_n u + J_{n-1} (u-av)v &\\
&= r(R_{n-3} u + R_{n-4} (u-av)v) &= r R_{n-2} 
 &= J_{n+1}.
\end{array}
\]
The penultimate equality holds as $a \neq 1$. Thus, the lemma is verified.
\end{proof}

\begin{proposition} \label{prop:(d)fora}
If $a \neq 0,1$, then $\ker \lambda_a$ is the ideal generated by the elements
\begin{align*} h_1 & :=e_1e_2^2 - e_1^2e_3 - (2a)e_2e_3 + (1+2a)e_1e_4 - (a^2+a)e_5, \\
h_2 & :=e_1e_5 - 4e_2 e_4+ 3e_3^2+2e_6 \\
h_3 & :=-4e_1^2e_2^2-4e_2^3+4e_1^3e_3+(20a^2+14a-7)e_3^2 - (16a^2+18a+5)e_1e_5+(16a^3+36a^2+16a-2)e_6.
\end{align*}
\end{proposition}

\begin{proof}
By Lemma~\ref{lem:lambdaa}, we just need to produce linearly independent elements of $\ker \lambda_a$ in degrees 5 and 6. We have by Routine~\ref{rout:kernel} that $\dim_\kk (\ker \lambda_a)_5 =1$ and that we can choose a basis of $(\ker \lambda_a)_5$ to be the element $h_1$. In fact, we
 verify that $\lambda_a(h_1) =0$ using Lemma~\ref{Zhang} and~\eqref{multiplyS}, while suppressing some $\mu$ superscripts:

{\small\[
\begin{array}{rl}
\lambda_a(h_1) &=  u(u-av)v(u-av)v - u^2(u-2av)v^2 - (2a)(u-av)v(u-2av)v^2\\
&\quad  + (1+2a)u (u-3av)v^3 - (a^2+a)(u-4av)v^4\\
&=x(x-ay)^{\mu}y(x-ay)^{\mu^3}y - xx^{\mu}(x-2ay)^{\mu^2}y^2 - (2a)(x-ay)y(x-2ay)^{\mu^2}y^2\\
&\quad  + (1+2a)x (x-3ay)^{\mu}y^3 - (a^2+a)(x-4ay)y^4\\
&=x(x-(1+a)y)y(x-(3+a)y)y - x(x-y)(x-(2+2a)y)y^2 - (2a)(x-ay)y(x-(2+2a)y)y^2\\
&\quad  + (1+2a)x (x-(1+3a)y)y^3 - (a^2+a)(x-4ay)y^4 ~=0.
\end{array}
\]
}

On the other hand, we have by Routine~\ref{rout:kernel} that $\dim_\kk (\ker \lambda_a)_6 =4$ and that we can take a basis of $(\ker \lambda_a)_6$ to be $h_2$, $h_3$ along with 
\begin{align*} 
h_4 & :=4e_2^3-4e_1e_2e_3+(7-4a)e_3^2+(1+4a)e_1e_5+(2-4a-4a^2)e_6, \\
h_5 & :=4e_2^3+(7-14a)e_3^2-4e_1^2e_4+(5+14a)e_1e_5+(2-16a-12a^2)e_6.
\end{align*}

\noindent By direct computation we have:
\[
\begin{array}{rl}
e_1h_1 & \hspace{-.1in}=e_1^2e_2^2 - e_1^3e_3 - (2a)e_1e_2e_3 + (1+2a)e_1^2e_4 - (a^2+a)e_1e_5,\\
h_1e_1 & \hspace{-.1in} =e_1e_2^2e_1 - e_1^2e_3e_1 - (2a)e_2e_3e_1 + (1+2a)e_1e_4e_1- (a^2+a)e_5e_1,\\
&\hspace{-.1in}=
e_1^2e_2^2
 - e_1^3e_3
-(2+2a)e_1e_2e_3
  +(2a)e_3^2
   +(3+2a)e_1^2e_4
  +(4a)e_2e_4
 -(2+7a+a^2)e_1e_5
 +4(a^2+a)e_6.\\
\end{array}
\]
\smallskip

\noindent {\bf Claim.} We have that  $h_2$, $h_3$, $e_1h_1$, $h_1e_1$ are $\kk$-linearly independent  and that 
$$h_4=2a(2a+1)h_2-h_3-(6+4a)e_1h_1+(2+4a)h_1e_1, \quad
h_5=4a^2h_2 -h_3-(4+4a)e_1h_1 +(4a)h_1e_1.$$

\vspace{.1in}

\noindent The proof is presented in the appendix; see Claim~\ref{aclaim2}. Therefore, the result holds.

Now for the reader's convenience, we verify that $\lambda_a(h_i) =0$ for $i=2,3$ using Lemma~\ref{Zhang} and~\eqref{multiplyS}, while suppressing some $\mu$ superscripts:

{\small\[
\begin{array}{rl}
\lambda_a(h_2) &= u(u-4av)v^4 - 4(u-av)v(u-3av)v^3 + 3(u-2av)v^2(u-2av)v^2 + 2(u-5av)v^5\\
&= x(x-4ay)^{\mu}y^4 - 4(x-ay)y(x-3ay)^{\mu^2}y^3 + 3(x-2ay)y^2(x-2ay)^{\mu^3}y^2 + 2(x-5ay)y^5\\
&= x(x-(1+4a)y)y^4 - 4(x-ay)y(x-(2+3a)y)y^3 + 3(x-2ay)y^2(x-(3+2a)y)y^2 + 2(x-5ay)y^5 ~=0;\\\\

\lambda_a(h_3) &= -4u^2(u-av)v(u-av)v-4(u-av)v(u-av)v(u-av)v+4u^3(u-2av)v^2\\
&\quad +(20a^2+14a-7)(u-2av)v^2(u-2av)v^2-(16a^2+18a+5)u(u-4av)v^4\\
&\quad +(16a^3+36a^2+16a-2)(u-5av)v^5\\
&=-4xx^{\mu}(x-ay)^{\mu^2}y(x-ay)^{\mu^4}y-4(x-ay)y(x-ay)^{\mu^2}y(x-ay)^{\mu^4}y+4xx^{\mu}x^{\mu^2}(x-2ay)^{\mu^3}y^2\\
&\quad +(20a^2+14a-7)(x-2ay)y^2(x-2ay)^{\mu^3}y^2-(16a^2+18a+5)x(x-4ay)^{\mu}y^4\\
&\quad +(16a^3+36a^2+16a-2)(x-5ay)y^5\\
&=-4x(x-y)(x-(2+a)y)y(x-(4+a)y)y-4(x-ay)y(x-(2+a)y)y(x-(4+a)y)y\\
&\quad +4x(x-y)(x-2y)(x-(3+2a)y)y^2 +(20a^2+14a-7)(x-2ay)y^2(x-(3+2a)y)y^2\\
&\quad -(16a^2+18a+5)x(x-(1+4a)y)y^4 +(16a^3+36a^2+16a-2)(x-5ay)y^5 ~=0.
\end{array}
\]}
\vspace{-.3in}

\end{proof}

\smallskip


\section{Elementary proofs that $U(W_+)$ and $U(W)$ are  not noetherian}\label{sec:easy}
In this section, we establish the remaining part of Theorem~\ref{ithm1}, that $\ker \lambda_0 = \ker \lambda_1$ is not finitely generated as a left or right ideal of $U(W_+)$.
We also prove Theorem~\ref{ithm:easy2}. 

We first focus on $U(W_+)$.  
Recall the map $\phi:  U(W_+) \twoheadrightarrow B$ from 
Definition~\ref{idef:maps}, and consider 
 Notations~\ref{inot0},~\ref{not:Q},~\ref{not1}, and~\ref{not2} along with the following.

\begin{notation}[$p$, $I$] \label{not3} Let  $p:= \phi(e_1 e_3-e_2^2-e_4)$ be an element of $B$, and let $I:= (p)$ be a two-sided ideal of $B$.  Note that by Proposition~\ref{prop:(d)}, $I = \phi(\ker \lambda_0 )= \pi_B(\ker \pi_0)$.
\end{notation}

We begin by establishing some basic facts about $p$ and $I$.

\begin{lemma} \label{lem:easy} We have the following statements:
\begin{enumerate}
\item $p= v^3w - v^2w^2$, 
\item $p$ is a normal element of $S$ and of $Q$,  and 
\item $I = Qp$.
\end{enumerate}
\end{lemma}

\begin{proof} We employ Lemma~\ref{Zhang} and~\eqref{multiplyS} in all parts.

(a) Consider the computation in $S$ below:
\[
\begin{array}{rl}
p =  \phi(e_1 e_3-e_2^2-e_4) &= u(u-2w)v^2 - (u-w)v(u-w)v - (u-3w)v^3\\
& = x(x-2z)^\mu y^{\mu^2}y^{\mu^3}- (x-z)y^{\mu}(x-z)^{\mu^2}y^{\mu^3} - (x-3z)y^{\mu}y^{\mu^2}y^{\mu^3}\\
& = x(x-y-2z)y^2 - (x-z)y(x-2y-z)y - (x-3z)y^3\\
& = y^3z - y^2z^2\\
& = v^3w - v^2w^2.
\end{array}
\]

(b) From part (a), we get that $p$ is a normal element of $S$, and of $Q$, since $vp = pv$, $wp= pw$, and
\[
\begin{array}{rll}
u p &= u(v^3w - v^2w^2) &= xy^{\mu}y^{\mu^2}y^{\mu^3}z^{\mu^4} -x y^{\mu}y^{\mu^2}z^{\mu^3}z^{\mu^4}\\
 &= (y^3z - y^2z^2)x &= (y^3z - y^2z^2)(x+4y)^{\mu^4}\\
  &= (v^3w - v^2w^2)(u+4v) &= p(u+4v).
\end{array}
\]

(c)  On one hand, we get that $I = BpB \subseteq QpQ = Qp$, by Lemma~\ref{lem:BinQ} and part (b). On the other hand, recall that $R$ is the subalgebra of $Q$ generated by $u, v$. We will  show by induction on $i$ and $j$ that $p(vw)^i R_{j-2i} \subseteq I$ for all $0 \leq i \leq \lfloor j/2 \rfloor$; this yields $pQ_j \subseteq I$.

The base case $i =j =0$ holds since $p \in I$. For the induction step, assume that $p(vw)^iR_{j-2i} \subseteq I$. Now it suffices to show that (i) $p(vw)^iR_{j+1-2i} \subseteq I$, and (ii) $p(vw)^{i+1}R_{j-2i} \subseteq I$.

For (i), we have by induction that $$I ~\supseteq~ up(vw)^iR_{j-2i} + p(vw)^iR_{j-2i}u=:I',$$ since $u$ is a generator of $B$. Now consider the following computations, where we suppress the action of $\mu$ on invariant elements and on graded pieces of $\kk[x,y]$:
\[
\begin{array}{rl}
I' &= x (y^3z - y^2z^2)(yz)^i \kk[x,y]_{j-2i} +  (y^3z - y^2z^2)(yz)^i\kk[x,y]_{j-2i} x^{\mu^{j+4}}\\
&= (y^3z - y^2z^2)(yz)^i  x\kk[x,y]_{j-2i} +  (y^3z - y^2z^2)(yz)^i(x+(j+4)y)\kk[x,y]_{j-2i}\\
&= (y^3z - y^2z^2)(yz)^i \left[ x\kk[x,y]_{j-2i} +  (x+(j+4)y)\kk[x,y]_{j-2i}\right]\\
&= (y^3z - y^2z^2)(yz)^i \kk[x,y]_{j+1-2i},\\
\end{array}
\]
where 
 the last equality holds since $j+4>0$. 
 Thus  (i) holds.

For (ii), we get that $p(vw)^iR_{j+2-2i} \subseteq I$ by applying (i) twice. Now
$$I ~\supseteq~ p(vw)^iR_{j+2-2i} +  p(vw)^iR_{j-2i}(uv-vw) 
~\supseteq~ p(vw)^iR_{j-2i}(vw). $$
Note that $ R_{k}(vw)  = (vw)R_{k}$ for all $k$.  
So    $I ~\supseteq~ p(vw)^{i+1}R_{j-2i}$ and we are done.
\end{proof}

Now we complete the proof of  Theorem~\ref{ithm1}.

\begin{theorem} \label{thm:easy}
The ideal $I$ of $B$ is not finitely generated as a left or right ideal.
As a result, the kernel of $\lambda_0$ is not finitely generated as a left or right ideal of $U(W_+)$. 
\end{theorem}

\begin{proof}
Recall that $\ker \lambda_0 = (e_1e_3-e_2^2-e_4)$ by Proposition~\ref{prop:(d)}. It is clear that if $\ker \lambda_0$ is finitely generated as a left/right ideal of $U(W_+)$, then $I$ is a finitely generated as a left/right ideal of $B$.
Therefore, to show that $\ker \lambda_0$ is not finitely generated it suffices to show that ${}_B I$ and $I_B$ are not finitely generated. 

By way of contradiction, suppose that ${}_B I$ is finitely generated. Then, there exists $n\geq 4$ so that $B I_{\leq n} = I$.  
Since $B$ is generated by $u$ and $(u-w)v$, 
we get that 
\begin{equation} \label{eq:BI}
I_{n+1} ~=~ B_1 I_n + B_2 I_{n-1} ~=~ u I_n + (u-w)v I_{n-1}.
\end{equation}
By Lemma~\ref{lem:easy}, 
 $I = Qp \subseteq SpS = Sp$.
Since $vI \subseteq vSp \subseteq Sp$, we get by \eqref{eq:BI} that 
\begin{equation} \label{eq:uSp}
I_{n+1} \subseteq uSp + (u-w)Sp = uSp + wSp.
\end{equation}
By using Lemma~\ref{Zhang} and \eqref{multiplyS}, it is easy to see that $uS + wS =   x \kk[x,y,z] + z \kk[x,y,z]$ and that a positive power of $y$ cannot belong to the right-hand-side. So, a positive power of $v$ cannot belong to $uS+wS$. Therefore,
\begin{equation} \label{eq:vp}
v^{n-3}p \not \in uSp + wSp.
\end{equation}
On the other hand, $v^{n-3}p \in I_{n+1}$ by Lemma~\ref{lem:easy}(c).
This contradicts~\eqref{eq:uSp} and \eqref{eq:vp}. Thus, ${}_B I$ is not finitely generated.

Next, suppose that $I_B$ is finitely generated. Then, there exists $n\geq 4$ so that $I_{\leq n}B = I$, with
\begin{equation} \label{eq:IB}
I_{n+1} ~=~ I_n B_1  + I_{n-1}B_2   ~=~ I_n u + I_{n-1} (u-w)v  ~=~ I_n u + I_{n-1} v(u + v-w).
\end{equation}
We get that $I, Iv \subseteq pS$ by Lemma~\ref{lem:easy}(b). So, the right-hand-side of \eqref{eq:IB} is contained in $pSu+pS(v-w)$. With an argument similar to that in the previous paragraph, $Su+S(v-w)$ does not contain positive powers of $v$. So, $p v^{n-3} \not \in I_n u + I_{n-1} v(u + v-w)$. On the other hand, $p v^{n-3} \in I_{n+1}$ by Lemma~\ref{lem:easy}(b,c), which contradicts \eqref{eq:IB}. Thus, $I_B$ is not finitely generated.
\end{proof}

\begin{remark}\label{rem:two}
We do not know whether or not $\ker \lambda_a$ is finitely generated for $a \neq 0,1$.
\end{remark}

One can of course deduce from Theorem~\ref{thm:easy} that $U(W)$ and $U(V)$ are neither left nor right noetherian; see \cite[Lemma~1.7]{SieWal:Witt}.  
Nevertheless,  a direct proof that $U(W)$ is not left or right noetherian is of independent interest, and we give such a result to end the section. 
First, we establish some notation.

\begin{notation}[$\widehat{S}$, $\widehat{R}$, $\widehat{B}$, $\widehat{\phi}$, $\widehat{\lambda}_a$, $\eta_a$, $\wh{I}$]\label{not:widehat}
Since $v$ is normal in $S$ and in $R$, we may invert it.  
Let $\widehat{S} := S[v^{-1}]$, and let $\widehat{R}:=R[v^{-1}]$.

Note that $\phi$ extends to  an algebra homomorphism $\widehat{\phi}:  U(W) \to \widehat{S} $  defined by \eqref{phi} for all $n \in \ZZ$. Likewise, $\lambda_a$ extends to an algebra homomorphism $\widehat{\lambda}_a:  U(W) \to \widehat{R}$ defined by \eqref{lambda} for all $n \in \ZZ$. 
For $a \in \kk$ define $\eta_a:  \widehat{S} \to \widehat{R}$ by $u \mapsto u, v\mapsto v, w \mapsto av$.
Note that $\widehat{\lambda}_a = \eta_a \widehat{\phi}$.

Let $\widehat{B} := \widehat{\phi}(U(W))$.
Finally, let $\wh{I} = \wh{\phi}(\ker \wh{\lambda}_0)$.  Note that $\wh{I} = \wh{B} \cap \ker \eta_0$.
\end{notation}

We first note that the proof of Lemma~\ref{lem1} extends to $U(W)$ to give
\beq \label{fact}
\ker \wh{\lambda}_0 = \ker \wh{\lambda}_1.
\eeq
We now show:  

\begin{proposition}\label{prop:easy2}
Recall $p =  \phi (e_1 e_3-e_2^2-e_4) =  w(v-w) v^2$ from Notation~\ref{not3} and Lemma~\ref{lem:easy}.
We have that
\[ \wh{I} =  \widehat{B} \cap \ker \eta_0 =  \widehat{B} \cap  \ker \eta_1= \widehat{B}p\widehat{B}  = \widehat{S}p = p\widehat{S}.\]
\end{proposition}

\begin{proof}
We first show that $\widehat{B}p\widehat{B}  = \widehat{S}p = p\widehat{S}$.
Certainly, $\widehat{B}p\widehat{B} \subseteq \widehat{S}p\widehat{S} = \widehat{S}p = p\widehat{S}$, where the last two equalities hold  because a normal element of $S$ will also be normal in $\widehat{S}$.

For the other direction, we will show that $\widehat{R} w^j p \subseteq \widehat{B}p\widehat{B}$ for all $j \geq 0$  by induction. Since  $\widehat{S} = \widehat{R} \cdot \kk[w]$, this will imply that $\wh{S} p \subseteq \wh{B}p\wh{B}$.
So assume that $w^j p \in \widehat{B}p\widehat{B}$ for some $j \geq 0$ (it is clear for $j=0$).  
Since $up=p(u+4v)$, we get that for all $n \in \ZZ$:
\[ \widehat{B}p\widehat{B} \ni ~[\widehat{\phi}(e_n), w^jp]= (u-(n-1)w)v^{n-1}w^j p - w^jp (u-(n-1)w)v^{n-1} =
(j+4) v^n w^j p.\]
So, $\kk[v,v^{-1}] \cdot w^j p \subseteq \widehat{B}p\widehat{B}$.
Since $u = \widehat{\phi}(e_1) \in \widehat{B}$, we have $\widehat{R} w^j p = \kk[u] \cdot \kk[v, v^{-1}] \cdot  w^jp \subseteq \widehat{B}p\widehat{B}$.
Finally,  since we have seen that $v^{-1} w^j p \in \widehat{R} w^j p \subseteq \widehat{B}p\widehat{B}$, we have that
\[ \widehat{B}p\widehat{B} \ni~ (\widehat{\phi}(e_1)- \widehat{\phi}(e_2) v^{-1}) w^j p = w^{j+1} p.\]
By induction,  $\wh{B}p\wh{B} = \wh{S} p$, as desired.

From the definitions,  $p \in (\ker \eta_0) \cap (\ker \eta_1)$.  So 
\[ \widehat{B}p\widehat{B} \subseteq (\ker \eta_0) \cap (\ker \eta_1) \cap \widehat{B} =w \widehat{S} \cap (v-w) \widehat{S} = w(v-w) \widehat{S} = p \widehat{S}.\]
Combining this with the first part of the proof, we have $\widehat{B}p\widehat{B} = (\ker \eta_0) \cap (\ker \eta_1) \cap \widehat{B}$.  By \eqref{fact} and the definition of $\wh{I}$, we have: 
\[ \wh{I}=(\ker \eta_0) \cap \widehat{B} = \widehat{\phi}(\ker \widehat{\lambda}_0)  = \widehat{\phi}(\ker \widehat{\lambda}_1) = (\ker \eta_1) \cap \widehat{B},\]
 completing the proof.
\end{proof}


From Proposition~\ref{prop:easy2} we obtain:

\begin{theorem}\label{thm:easy2}
The ideal $\wh{I}$ of $\wh{B}$ is not finitely generated as a left or right ideal. 
As a result, the kernel of $\widehat{\lambda}_0$ is not finitely generated as a left or right ideal
of $U(W)$.
\end{theorem}

\begin{proof}
 This argument is similar to the proof of Theorem~\ref{thm:easy}.
 It suffices to show that $\widehat{I} $ is not finitely generated as a left or right ideal of $\widehat{B}$.   

By way of contradiction, suppose that  for some $n \in \NN$, we have $\widehat{I} = \widehat{B} (\widehat{I}_{-n} \oplus \dots \oplus \widehat{I}_n)$.  For all $k \in \ZZ$, we have $\widehat{\phi}(e_k) \in u \widehat{S} + w \widehat{S}$. So, $\widehat{B}_k \subseteq u\widehat{S} + w\widehat{S}$ for all $k \neq 0$, and $\wh{I}_k \subseteq u \wh{S}+w \wh{S}$ for all $k$ with $\abs{k} > n$.   Note that a power of $v$ cannot belong to $u\wh{S} + w\wh{S}$. So, $v^{n-3}p \not \in  \wh{I}$. 
However, by Proposition~\ref{prop:easy2}, we get that $\wh{I} = \wh{S}p$ and $v^{n-3}p \in \widehat{I}$.  This contradiction shows that ${}_{\widehat{B}} \widehat{I}$ is not finitely generated. 

The proof that $\widehat{I}_{\widehat{B}}$ is not finitely generated is similar; we leave the details to the reader.
\end{proof}

\begin{corollary}\label{cor:Vir}
The universal enveloping algebra $U(V)$ is neither left nor right noetherian.
\end{corollary}
\begin{proof}
This follows directly from Theorem~\ref{thm:easy2}, since $U(W) = U(V)/(c)$.
\end{proof}

\begin{remark}\label{rem:one}
After the first draft of this paper was finished, we learnt of the results of Conley and Martin in \cite{CM07}. We thank the referee for calling \cite{CM07} to our attention.
The paper considers a family of homomorphisms defined as (using their notation) 
\[ \pi_\gamma: U(W) \to \kk[x,x^{-1}, \partial], \quad e_n \mapsto x^{n+1} \partial + (n+1) \gamma x^n.\]
Using the identification $u = x^2 \partial$, $v =x$ from the discussion after Theorem~\ref{ithm1}, we have
\[ \wh{\lambda}_a (e_n) = (x^2 \partial - (n-1) a x) x^{n-1} = x^{n+1} \partial + (1-a)(n-1)x^n.\]
The reader may verify that $ \wh{\lambda}_a(e) = x^{2(1-a)} \pi_{1-a}(e) x^{-2(1-a)}$ for all  $e \in U(W)$ (where here one uses a suitable extension of $k[x,x^{-1}, \partial]$ to carry out computations).  As a result,
\beq\label{boff}
\ker \wh{\lambda}_a = \ker \pi_{1-a}
\eeq
for all $a \in \kk$.

\cite[Theorem~1.2]{CM07} shows (using \eqref{boff}) that $\ker \wh{\lambda}_0 = \ker \wh{\lambda}_1 = (e_{-1}e_2-e_0e_1-e_1)$.
Recall from Proposition~\ref{prop:(d)} that $\ker \lambda_0 $ is generated as  a two-sided ideal by $g_4:= e_1e_3 - e_2^2-e_4$.
 A computation gives that
\[ \ad(e_{-1}^3)(g_4) = [e_{-1}, [e_{-1}, [e_{-1}, g_4]]] = 12(e_{-1}e_2-e_0e_1 -e_1),\]
and it follows that 
\[ (g_4) = \ker \wh{\lambda}_0 = \ker \wh{\lambda}_1 = (e_{-1} e_2 - e_0e_1 -e_1).\]
\end{remark}


\section{The connection between the maps $\phi$ and $\rho$} \label{sec:kernels} 

For the remainder of the paper, we return to considering $U(W_+)$.  
The main goal of this section is to relate the map $\phi$ (of Definition~\ref{idef:maps}) that played a crucial role in the proof of Theorem~\ref{thm:easy}  to the map $\rho$ (of Notation~\ref{inot1}) that was the focus of \cite{SieWal:Witt}.  We show that $\ker \phi = \ker \rho$; in fact, we have:

\begin{theorem} \label{thm:kernels}
We have that $\ker \rho = \ker \phi = \bigcap_{a \in \kk} \ker \lambda_a$.  As a consequence, $\rho(U(W_+)) \cong \phi(U(W_+))$.
\end{theorem}

Consider Notation~\ref{inot0} and the following notation for this section.  Recall the definitions of $X, f , \tau$ from Notation~\ref{inot1}. So, $\tau \in \Aut(X)$ and $\tau^*:\kk(X) \to \kk(X)$ is the pullback of $\tau$.
Here we take $\mu \in \Aut(\PP^2)$ and $\nu \in \Aut(\PP^1)$ to be morphisms of varieties, 
defined by
\[ \mu([x:y:z]) = [x-y:y:z] \quad \mbox{ and } \quad \nu([x:y]) = [x-y:y].\]
We 
denote the respective pullback morphisms  by $\mu^*$ and $\nu^*$.  However,   to be consistent with Lemma~\ref{Zhang} (and abusing notation slightly), we still write $$S \cong \kk[x,y,z]^{\mu} \quad \text{ and } \quad R \cong \kk[x,y]^{\nu}.$$
We also establish the convention that $h^\tau := \tau^* h$ for $h \in \kk(X)$, and similarly for pullback by other morphisms.

\smallskip

Before proving Theorem~\ref{thm:kernels}, we provide some preliminary results.

\begin{lemma}[$\psi_a$, $\Psi_a$] \label{lem:psi}
For $a \in \kk$, we have the following statements.
\begin{enumerate}
\item We have a well-defined morphism $\psi_a:  \PP^1 \to X$ given by 
\[\psi_a([x:y]) = [2x^2-4xy-6ay^2: x^2-2xy+y^2: -x^2+3xy-2y^2: x^2-4xy+4y^2].\]
\item $\psi_a \nu = \tau \psi_a.$
\item $\psi_a^*$ extends to an algebra homomorphism $ \Psi_a:  \kk(X)[t; \tau^*] \to \kk(\PP^1)[s; \nu^*]$, where $\Psi_a(t) = s$.
\end{enumerate}
\end{lemma}

\begin{proof}
(a,b) Both are straightforward.  Part (a) is a direct computation.    
In   Section~\ref{M2} in the appendix, we verify that $ (\psi_a \nu)^* = \nu^* \psi_a^* = \psi_a^* \tau^* = (\tau \psi_a)^*$ as maps from $\kk(X) \to \kk(\PP^1)$.  Thus, (b) holds.  

(c) We have for all $h, \ell \in \kk(X)$ and $n, m \in \NN$ that:
\[ 
\begin{array}{rlll}
\Psi_a(h t^n \ell t^m) 
&= \Psi_a(h \ell^{\tau^n} t^{n+m} )
&= \psi_a^*(h) \psi_a^*(\ell^{\tau^n}) s^{n+m}&\\
&= \psi_a^*(h)  \psi_a^* (\ell)^{\nu^n} s^{n+m} 
&= \psi_a^*(h)  s^n \psi_a^* (\ell )s^{m} 
&= \Psi_a(h t^n) \Psi_a(\ell t^m).
\end{array}
\]
Thus,  $\Psi_a$ is an algebra homomorphism.

\vspace{-.3in}

\end{proof}

\vspace{.05in}

\begin{lemma}[$C_a$] \label{lem:C}
For $a \in \kk$, define the curve 
$$C_a = V(w+6ax+(4+12a)y+(2+6a)z, xz-y^2)~~ \subseteq X.$$
Then, $\psi_a$ defines an isomorphism from $\PP^1 \to C_a$.
\end{lemma}

\begin{proof}
That the image of $\psi_a$ of Lemma~\ref{lem:psi}(a) is contained in $C_a$ is a straightforward verification.  The inverse map to $\psi_a$ is defined by
the birational map $[w:x:y:z]  \mapsto [2x+y:x+y]$; we leave the verification of the details to the reader. 
\end{proof}

\begin{lemma}[$\gamma$] \label{lem:gamma}
Define a map $\gamma:  R \to \kk(\PP^1)[s; \nu^*]$ as follows:  if $h  \in R_n = \kk[x,y]_n$, let
\[ \gamma(h) = \frac{h}{x(x-y)\cdots(x-(n-1)y)}s^n .\]
Then, $\gamma$ is an injective $\kk$-algebra homomorphism.
\end{lemma}

\begin{proof}
Let $h \in \kk[x,y]_n$ and $\ell \in \kk[x,y]_m$.
Then,
\[
\begin{array}{rl}
\medskip

 \gamma(h\ast \ell) &= \gamma (h \ell^{\nu^n}) =  \displaystyle \frac{h \ell^{\nu^n}}{x (x-y)\cdots (x-(n+m-1)y) } s^{n+m} \\
 \medskip
 
&=  \displaystyle \frac{h}{x (x-y) \cdots (x-(n-1)y)} \left(\frac{\ell}{x (x-y) \cdots (x-(m-1)y)}\right)^{\nu^n} s^{m+n} \\
 &=  \displaystyle \frac{h}{x (x-y) \cdots (x-(n-1)y)} s^n \frac{\ell}{x (x-y) \cdots (x-(m-1)y)} s^m 
 =\quad  \gamma(h) \gamma(\ell).
 \end{array}
 \]
So, $\gamma$ is a homomorphism; injectivity is clear.
\end{proof}

\begin{proposition}\label{prop:pullback} Retain the notation of Lemmas~\ref{lem:psi} and~\ref{lem:gamma}.
Let $a \in \kk$.  Then, $\Psi_a \rho = \gamma \lambda_a$ as maps from $U(W_+) \to \kk(\PP^1)[s; \nu^*]$, and $\ker \Psi_a \rho = \ker \lambda_a$.
\end{proposition}

\begin{proof}
By Lemma~\ref{lem:present}(a), it suffices verify that the maps $\Psi_a \rho$ and $\gamma \lambda_a$ agree on $e_1$ and $e_2$.  We have:
\[ \Psi_a( \rho(e_1)) = \Psi_a(t)  = s = \gamma(u)  = \gamma( \lambda_a(e_1)).\]
We verify that
\beq\label{two}
\psi_a^*(f)=  \frac{xy-ay^2}{x^2-xy}
\eeq 
in Section~\ref{M2} in the appendix.
Thus, 
\[ \Psi_a(\rho(e_2) )= \psi_a^*(f)s^2 
= \frac{xy-ay^2}{x^2-xy} s^2 = \gamma(uv-av^2) = \gamma \lambda_a(e_2).
\]
The final statement follows from the fact that $\gamma$ is injective (Lemma~\ref{lem:gamma}).
\end{proof}

We now prove Theorem~\ref{thm:kernels}.

\begin{proof}[Proof of Theorem~\ref{thm:kernels}]
By Lemma~\ref{lem:C}, $\psi_a^* h = 0$ if and only if $h |_{C_a} \equiv 0$.
Now, the curves $C_a$ cover an open subset of $X$.
(One way to see this is that, because $\bigcup_a C_a$ is  dense in $X$ and is clearly constructible, by \cite[Exercise II.3.19(b)]{Hartshorne} it contains an open subset of $X$.)
Thus if $h \in \kk(X)$ is in the intersection $\bigcap_a \ker \psi_a^*$, then $h$ vanishes on this open subset and so $h=0$.  
So, $\bigcap_a \ker \Psi_a = \{0\}$. Thus, $\ker \rho = \bigcap_a \ker \Psi_a \rho = \bigcap_a \ker \lambda_a$, 
where the last equality holds by Proposition~\ref{prop:pullback}.

To show that $\ker \phi = \bigcap_a \ker \lambda_a$, define  closed immersions $i_a:  \PP^1 \to \PP^2$ for $a \in \kk$ by $i_a([x:y]) = [x:y: ay]$.  Then, $\im (i_a) = V(z-ay)$, and pullback along $i_a$ induces the ring homomorphism
\[ i_a^*:  \kk[x,y,z] \to \kk[x,y] \quad x \mapsto x,~~ y \mapsto y, ~~z \mapsto ay.\]
The reader may verify that  $i_a \nu = \mu i_a$, and that $i_a^*$ is also  a homomorphism from 
$S = \kk[x,y,z]^{\mu}$ to $R=\kk[x,y]^{\nu}$.  In terms of $u,v,w$, we have
\[ i_a^*(u) = u, \quad i_a^*(v) = v, \quad i_a^*(w) = av.\]
That is, $i_a^* = \eta_a |_S$, where $\eta_a$ was defined in Notation~\ref{not:widehat}.
We see that $i_a^* \phi = \lambda_a$.

As with the first paragraph, the curves $V(z-ay)$ cover an open subset of $\PP^2$:  in fact, $\bigcup_a V(z-ay) \supseteq (\PP^2 \ssm V(y))$.  So $ \bigcap_a \ker i_a^*=\{0\}$.
Thus, 
$ \ker \phi = \bigcap_a \ker i_a^* \phi = \bigcap_a \ker \lambda_a$,
completing the proof.
\end{proof}


\section{The kernel of  $\phi$}\label{PHI}

In this section, we analyze the map $\phi$ from Definition~\ref{idef:maps}.
In particular, we verify part (c) of Theorem~\ref{ithm2}(c). To proceed, recall Notations~\ref{inot0},~\ref{not:Q},~\ref{not1}, and~\ref{not2}.

\begin{theorem} \label{thm:(b)}
The kernel of $\phi$ is generated as a two-sided ideal by $g:=e_1e_5-4e_2e_4+3e_3^2+2e_6$.
\end{theorem}

\begin{proof}
First, observe that as $e_1e_5, e_2e_4, e_3^2, e_6$ are elements of the standard basis for $U(W_+)$ (by Lemma~\ref{lem:present}(b)), they  are linearly independent. So, we have that $g \neq 0$.  

Now we verify that $\phi(g) = 0$ by using Lemma~\ref{Zhang} and~\eqref{multiplyS}:
\[
\begin{array}{rl}
\phi(g) &= u(u-4w)v^4 - 4(u-w)v(u-3w)v^3 +3(u-2w)v^2(u-2w)v^2 + 2(u-5w)v^5\\
&=x(x-4z)^{\mu}y^4 - 4(x-z)y(x-3z)^{\mu^2}y^3 + 3(x-2z)y^2(x-2z)^{\mu^3}y^2 + 2(x-5z)y^5\\
&=x(x-y-4z)y^4 - 4(x-z)y(x-2y-3z)y^3 + 3(x-2z)y^2(x-3y-2z)y^2 + 2(x-5z)y^5 ~=0.
\end{array}
\]
\smallskip

We take the following notation for the rest of the proof.

\begin{notation}[$M$, $M'$, $b_5$, $b_6$, $b_7$, $\eta$] \label{not:M}
Consider the right $B$-modules
$$M:=uB \cap (u-w)vB \quad \text{ and } \quad M':=b_5B+b_6B+b_7B,  ~~\text{with}$$
\[
\begin{array}{rl}
b_5 &= (uv-vw)(u^3-6(uv-vw)u +12u(uv-vw)),\\
b_6 &= (uv-vw)(-48(uv-3vw)v^2-36u(uv-2vw)v+u^4),\\
b_7 &= (uv-vw)(u^5-40((uv-vw)^2u-3(uv-vw)u(uv-vw)+3u(uv-vw)^2)).
\end{array}
\]

\vspace{.05in}

\noindent Moreover, take $\eta:B \to A(0)$ to be the map induced by the projection $\eta_0: \wh{S} \twoheadrightarrow \wh{R} = \wh{S}/(w)$ from Notation~\ref{not:widehat}. 
\end{notation}

The remainder of the proof will be established through a series of lemmas. 

\begin{lemma}\label{claim1}
We obtain that $b_5, b_6, b_7\in uB \cap (u-w)vB$. In other words, $M' \subseteq M$.
\end{lemma}

\begin{proof}
Let
\begin{eqnarray}
r_5 &:=& e_2(e_1^3-6 e_2 e_1+12 e_1 e_2),  \label{eq:r5} \\
r_6 &:=& e_2(-48 e_4 - 36 e_1 e_3 + e_1^4), \label{eq:r6}\\
r_7 &:=& e_2(e_1^5-40(e_2^2 e_1 - 3 e_2 e_1 e_2 + 3 e_1 e_2^2)).  \label{eq:r71}
\end{eqnarray}
We have as a consequence of the degree 5 relation of $U(W_+)$ in  Lemma~\ref{lem:present}(a) that
\begin{equation}  \label{eq:r51}
r_5 = e_1(e_1^2e_2-3 e_1 e_2 e_1+3e_2 e_1^2+6 e_2^2) ,
\end{equation}
and as a consequence of the degree 7 relation of $U(W_+)$ in  Lemma~\ref{lem:present}(a) that
\begin{equation} \label{eq:r7}
r_7 = e_1(e_1^4 e_2 - 5 e_1^3 e_2 e_1 + 10 e_1^2 e_2 e_1^2-10e_1 e_2 e_1^3+5 e_2 e_1^4-40 e_2^3).
\end{equation}
Thus $r_5, r_7 \in e_1 U(W_+) \cap e_2 U(W_+)$.  Since $b_5 = \phi(r_5)$ and $b_7 = \phi(r_7)$, these are both in $uB \cap (uv-vw)B$.

Note that $r_6 \in e_2 U(W_+)$, so $b_6 = \phi(r_6) \in (u-w)vB$.  Further,
\[ r_6 = e_1 (-36e_2e_3-18e_5+2e_4e_1-e_3e_1^2+e_2e_1^3)+12g.\]
Thus, $b_6 \in uB$ as well.
\end{proof}

\begin{lemma}\label{newclaim}
Suppose that 
$M = M'$.
Then, $\ker \phi = (g )$ and the theorem holds.
\end{lemma}

\begin{proof}
Let $K$ be the kernel of 
\begin{align*}
\alpha:  B[-1] \oplus B[-2] & \to B \\
(b,b') & \mapsto (ub+(uv-vw)b').
\end{align*}
It is a standard fact  that the map
\[ \beta:  M \to K\]
defined by $\beta( r) =  (u^{-1}r, -(uv-vw)^{-1} r)$
is an isomorphism of graded right $B$-modules, as in the proof of Lemma~\ref{towards(d)}.
Thus, $K$ is generated by $\beta(b_5)$,  $\beta(b_6)$, and $\beta(b_7)$ by the assumption.
By Proposition~\ref{propappendix} in the appendix, the kernel of $\pi_B$ is generated as a 2-sided ideal of $\kk\ang{t_1,t_2}$ by a degree 5 element $q_5$, a degree 6 element $q_6$, and a degree 7 element $q_7$.  We compute $q_5$ and $q_7$ by applying the formula from Proposition~\ref{propappendix} to $\beta(b_5)$ and $\beta(b_7)$, and by using \eqref{eq:r5}-\eqref{eq:r7}. Namely,  take 
\[
\begin{array}{lll}
\tilde{b}^1_1 = t_1^2t_2-3t_1t_2t_1+3t_2t_1^2+6t_2^2, && \tilde{b}^1_2 = -t_1^3+6t_2t_1-12t_1t_2\\
\tilde{b}^2_1 = t_1^4t_2-5t_1^3t_2t_1+10t_1^2t_2t_1^2-10t_1t_2t_1^3 +5t_2t_1^4-40t_2^3 && \tilde{b}^2_2=-t_1^5+40(t_2^2t_1-3t_2t_1t_2+3t_1t_2^2).

\end{array}
\]
So, we have that 
\[
\begin{array}{l}
q_5 = t_1 \tilde{b}^1_1 + t_2 \tilde{b}^1_2 = [t_1,[t_1, [t_1,t_2]]]]+6[t_2,[t_2,t_1]],\\
 q_7 =t_1 \tilde{b}^2_1 + t_2 \tilde{b}^2_2 = [t_1,[t_1,[t_1,[t_1,[t_1,t_2]]]]]+40[t_2,[t_2,[t_2,t_1]]].
 \end{array}
 \]

By Lemma~\ref{lem:present}(a), $q_5$ and $q_7$ generate the kernel of $\pi$. So, 
$ \ker \phi = \pi(\ker \pi_B) = ( \pi(q_6))$.  We see immediately that $(\ker \phi)_6$ is a 1-dimensional $ \kk$-vector space, generated by $\pi(q_6)$.  
Since $g \in (\ker \phi)_6$ is nonzero,  we have $g = \pi(q_6)$ up to scalar multiple. Therefore, $\ker \phi = ( g)$.
\end{proof}

Our goal now is show that $M=M'$; we do this by comparing Hilbert series.  To proceed, we show that:

\begin{lemma}\label{claim2}
 The Hilbert series of $M$ is $t^5(1-t)^{-2}(1-t^2)^{-1}$.
\end{lemma}

\begin{proof}
Since $A(0) = \kk \oplus u R$ we have 
\[ \hilb A(0) = 1 + t (\hilb R) = 1+\frac{t}{(1-t)^2} = \frac{1-t+t^2}{(1-t)^2}.\]

On the other hand, it is well-known that $\hilb Q = \hilb \kk[x,y,yz] =  (1-t)^{-2}(1-t^2)^{-1}$. Since $\lambda_0 = \eta \circ \phi$, we get that $\ker \eta = \phi(\ker \lambda_0)$ (which is denoted by $I$ in Notation~\ref{not3}).
So, by Lemma~\ref{lem:easy}(c),  we get
\[ \hilb \ker \eta  = \frac{t^4}{(1-t)^{2}(1-t^2)}.\]
Then 
\[ \hilb B = \hilb A(0) + \hilb \ker \eta  = \frac{1-t+t^3-t^4}{(1-t)^2(1-t^2)} + \frac{t^4}{(1-t)^2(1-t^2)} = \frac{1-t+t^3}{(1-t)^2(1-t^2)}.\]

Finally, we compute $\hilb M$ from the exact sequence
\[ \xymatrix{
0 \ar[r] &M \ar[r]^(0.3){\beta} &B[-1]  \oplus B[-2] \ar[r]^(0.7){\alpha} &  B \ar[r]& \kk \ar[r]& 0,}\]
where $\alpha, \beta$ are as in the proof of Lemma~\ref{newclaim}.
This gives
\[ \hilb M = (t^2+t-1) (\hilb B) + 1 =\frac{t^5}{(1-t)^2(1-t^2)},\]
as claimed.
\end{proof}

We now provide results on the Hilbert series of $M'$.  

\begin{lemma}\label{newclaim4}
We have that $\hilb \eta(M') \geq t^5(1-t)^{-2}$.
\end{lemma}
\begin{proof}
Let $a_5 := \eta(b_5)$ and $a_6 := \eta(b_6)$.  Then,
\[
\begin{array}{rl}
a_5 & = uvu(u^2-6vu+12uv) = xy(x-2y)[ (x-3y)(x-4y) - 6 y(x-4y)+12(x-3y)y] \\
& = x^2(x-y)(x-2y)y.\\
a_6 & = uvu(u^3-36uv^2-48v^3)  = xy(x-2y)[(x-3y)(x-4y)(x-5y) - 36 (x-3y)y^2-48y^3]\\
& = x^2 (x-y)(x-2y)y(x-11y)\\
& = a_5 (u-6v).
\end{array}
\]
Since $a_5u$ and $a_5(u-6v)$ are in $\eta(M')$ and $u$ and $u-6v$ span $R_1$, we have 
$a_5 R_1 \subseteq \eta(M')$. 
We get that 
$ \eta(M') \supseteq a_5 A(0) + a_5 R_1 A(0)$, as $\eta(M') $ is a right $A(0)$-module and contains $a_5 R_{\leq 1}$. Since $A(0) + R_1A(0)=R$, we obtain that $\eta(M') \supseteq a_5R$.
Now  as $\hilb R = (1-t)^{-2}$, we conclude that $\hilb \eta(M') \geq t^5(1-t)^{-2}$.
\end{proof}

\begin{lemma}\label{finalclaim} 
We have that $\hilb(M' \cap \ker \eta) \geq t^7(1-t)^{-2}(1-t^2)^{-1}$. 
\end{lemma}
\begin{proof} 
Again, recall that $\ker \eta = \phi(\ker \lambda_0)$, which is denoted by $I$ in Notation~\ref{not3}. Moreover by Lemma~\ref{lem:easy}(c), we have that $I = Qp = pQ$, where $p  = v^3w - v^2w^2$.  
Let 
\[ h := (uv-vw)(u+2v)p = (xy-yz)x (y^3z-y^2z^2).\]

Now we proceed by asserting the following: 
\medskip

\noindent {\bf Claim}.
$b_5 Q + b_6 Q + b_7 Q \ni ~x(xy-yz)(xyz+y^2z) = (uv-vw)(u+2v)(u+4v)vw.$
\medskip

\noindent The proof of this claim is provided in the appendix; see Claim~\ref{claim:final}(a).

Since $M' \cap I \supseteq M' I = b_5 Qp + b_6  Qp + b_7  Qp$,
we have 
\beq \label{bar} 
M' \cap I \supseteq (uv-vw)(u+2v)(u+4v)vwpQ = (xy-yz)x(y^3z-y^2z^2)(x+y)yz Q = h(x+y)yz Q. \eeq

We now show by induction that $ M' \cap I \supseteq h Q_n$ for all $n \geq 0$.  

\medskip

\noindent {\bf Claim}. $M' \cap I \supseteq h Q_n$ for $n =0,1,2$.
\medskip

\noindent The proof of this assertion is provided in the appendix; see Claim~\ref{claim:final}(b). We will prove the result for larger $n$ by geometric arguments.  
The maximal graded non-irrelevant ideals of $\kk[x,y,yz]$ are in bijective correspondence with  $\kk$-points of the weighted projective plane $\PP(1,1,2)$ \cite[Example~10.27]{Harris}.
We use the notation $(a:b:c)$ to denote a point of $\PP(1,1,2)$.
Let $$K(n) := (x-ny)\kk[x,y,yz] + (y^2-yz)\kk[x,y,yz],$$ be the graded ideal of polynomials vanishing at $(n:1:1)$.

Suppose now that $ M' \cap I \supseteq h Q_n$ for some $n \geq 2$.
Then, $M' \cap I$ contains
\[
\begin{array}{rl}
h \left[ Q_n u + Q_{n-1} (uv-vw) \right] &=
h \left[(x-(n+7)y) \kk[x,y,yz] + ( (x-(n+6)y)y - yz) \kk[x,y,yz]\right]_{n+1}\\
&=h \left[  (x-(n+7)y) \kk[x,y,yz] + (y^2-yz) \kk[x,y,yz] \right]_{n+1}\\
& =  h K(n+7)_{n+1}.
\end{array}
\]

From \eqref{bar}, we get that $(M'\cap I)_{n+1} \ni h (xyz+y^2z) y^{n-2}$.
Since $(xyz+y^2z) y^{n-2}$ does not vanish at $(n+7:1:1)$, it is not in   $h K(n+7)_{n+1}$.
Thus,
\[  h K(n+7)_{n+1} + \kk h (xyz+y^2z) y^{n-2} = h \kk[x,y,yz]_{n+1} \subseteq M' \cap I,\]
where the equality holds as  $h K(n+7)_{n+1}$ is codimension 1 in $h \kk[x,y,yz]_{n+1}$. Hence, $h Q_{n+1} \subseteq M' \cap I$.

Now by induction, we obtain that $M'\cap I \supseteq hQ$. Since $\hilb Q = (1-t)^{-2}(1-t^2)^{-1}$, we have 
\[ \hilb (M' \cap I ) \geq \frac{t^7}{(1-t)^2(1-t^2)}.\]

\vspace{-.25in}
\end{proof}

Our final lemma is 

\begin{lemma}\label{lem:final}
We have that $\hilb M = \hilb M' = t^5(1-t)^{-2} (1-t^2)^{-1}$.
As a result, $M = M'$.

\end{lemma}

\begin{proof}
Combining Lemmas~\ref{newclaim4} and \ref{finalclaim}, we have
\[ \hilb(M') \geq \frac{t^5}{(1-t)^2} + \frac{t^7}{(1-t)^{2}(1-t^2)} = \frac{t^5}{(1-t)^{2}(1-t^2)}.\]
On the other hand, by Lemmas~\ref{claim1} and~\ref{claim2} we get that
\[ \hilb(M') \leq \ \frac{t^5}{(1-t)^{2}(1-t^2)}.\]
Thus, $\hilb M  =\hilb M'$. Since $M' \subseteq M$ again by Lemma~\ref{claim1}, we conclude that  $M = M'$.
\end{proof}

Theorem~\ref{thm:(b)} now follows from Lemmas~\ref{newclaim} and~\ref{lem:final}.
\end{proof}

\begin{remark}\label{rem:three}
A shorter proof of Theorem~\ref{thm:(b)} follows from the results of \cite{CM07}.  Recall from Notation~\ref{not:widehat} that we may extend $\phi$ to a map $\wh{\phi}:  U(W) \to \wh{S}$, using the same formula  \eqref{phi} for $\wh \phi (e_n)$ with $n \leq 0$.  Then \cite[Theorem~1.3]{CM07} and \eqref{boff}, together with Theorem~\ref{thm:kernels}, give that $\ker \wh{\phi} = (e_{-1} e_3 -4 e_0 e_2 + 3 e_1^2+ 2e_2)$.  The reader may verify that
\[ \ad(e_{-1}^4)(g) = [e_{-1},[e_{-1},[e_{-1},[e_{-1},g]]]] = 24(e_{-1} e_3 -4 e_0 e_2 + 3e_1^2+2 e_2).\]
Since $\wh{\phi}(g) =0$, we have $(g) \subseteq \ker \wh{\phi} = (e_{-1} e_3 -4 e_0 e_2 + 3e_1^2+2 e_2) \subseteq (g)$, so all are equal.
\end{remark}


\section{A partial result on chains of two-sided ideals}\label{ACC}

It is not known whether $U(W_+)$ satisfies the ascending chain condition (ACC) on two-sided ideals; see Question~\ref{ques:ACC}. We do not answer this question here; however, we prove the partial result that the non-noetherian factor $B$ of $U(W_+)$ does have ACC on two-sided ideals.

Recall Notations~\ref{inot0},~\ref{not:Q}, \ref{not1}; in particular, $Q$ is the subalgebra of $S$ generated by $u, v, vw$.
Throughout, we consider $B$ as a subalgebra of $Q$.  
We begin by proving:
\begin{lemma}\label{lem:normalfact}
Let $h$ be a nonzero, homogeneous, normal element of $Q$,
and let $a \in \kk$.
Then, the $Q$-bimodules
\[N := hQ/hvQ \quad \text{and} \quad M_{a} = hQ / h(vw-a v^2)Q \]
are 
 noetherian $B$-bimodules under the action induced from $Q$.
 \end{lemma}

\begin{proof}
We remark that any normal element of $Q$ must be in the commutative subalgebra $\kk[v,vw]$, and thus, must commute with $v$ and $vw$.  
In particular, 
$vQ N  = 0$
 and $(vw-av^2)Q M_a = 0= M_a (vw-av^2)Q $.  

Let $\theta:  Q \to Q/vQ$ be the canonical projection.
 (Note that $vw \not \in \ker \theta$.) Since $u(vw)-(vw)u = 2 v^2w$ is contained in $\ker \theta$, the image $Q/vQ$ is commutative. It is easy to see that $Q/vQ \cong \kk[s,t]$ under the identification $s = \theta(u)$, $t = \theta(vw) = \theta(uv-vw)$.
Note that $s = \theta(\phi(e_1))$ and $t = \theta(\phi(e_2))$ 
are in $B$. So, $\theta(B) = Q/vQ$.
Thus, a left  $B$-submodule of $hQ/hvQ$ is simply an ideal of $\kk[s,t]$. So, $hQ/hvQ$ is noetherian as a left $B$-module.  
As chains of $B$-bimodules are also chains of left $B$-modules, $hQ/hvQ$ is also a noetherian $B$-bimodule.

Now define an algebra homomorphism $\delta:  Q \to R$ by $\delta(u) = u$, $\delta(v) = v$, and $\delta(vw) = a v^2$. (Note that $\delta = \eta_a|_Q$ from Notation~\ref{not:widehat}.)
It is easy to see that $\ker \delta = (vw-a v^2)Q$ and that $\delta$ is surjective.
Note also that $\delta(\phi(e_1))  = u$ and  $\delta(\phi(e_2)) = uv-a v^2$.  
Thus, $\delta(B) = A(a) $ as subalgebras of $R$.
If $a \neq 0,1$, then by Proposition~\ref{prop:(c)}, $A(a) \supseteq R_{\geq 4}$ is noetherian, and $R$ is a finitely generated right $A(a)$-module.   
If $a = 0$, then $R = A(0) + vA(0)$ is again a finitely generated right $A(0)$-module, and $A(0)$ is noetherian.  
Thus for $a \neq 1$, $M_a$ is also a finitely generated right $A(a)$-module. So, $M_a$ is noetherian as a  right $B$-module, let alone a $B$-bimodule.

If $a = 1$ then we have, similarly, that $ \delta(B) = A(1)$ is noetherian, and that $R = A(1) + A(1) v$ is a finitely generated left $A(1)$-module.   It follows that $M_a$ is a finitely generated left $A(a)$-module. So, $M_a$ is noetherian as a left $B$-module, and again as a $B$-bimodule.
\end{proof}

We now use geometric arguments to show:
\begin{proposition}\label{prop:noeth2}
Suppose that $\kk$ is algebraically closed, and let $K \subseteq Q$ be  a nonzero graded ideal.  Then, $Q/K$ is a noetherian $B$-bimodule.
\end{proposition}

\begin{proof}
Let $T$ be the commutative ring $\kk[x,y,yz]$.  
We consider $K$ as a subset of $T$, since (via Lemma~\ref{Zhang}) $Q = T^\mu$ and $T$ have the same underlying vector space.  
For all $n, m \in \NN$, we have
\beq \label{star1}
K_{n+m} \supseteq K_n Q_m = K_n (T_m)^{\mu^n} = K_n T_m,
\eeq
and so $K$ is also an ideal of $T$.
Further,
\beq \label{star2}
K_{n+m} \supseteq Q_m K_n = T_m (K_n)^{\mu^m}.
\eeq
If $T$ were generated in degree 1,  one could obtain directly from \eqref{star1}, \eqref{star2} that  $K_n$ is $\mu$-invariant for $n \gg 0$ (or see \cite[Lemma~4.4]{AS}).  A similar statement holds in our case; however, a proof would take us too far afield so we work more directly with the graded pieces of $K$.

Choose $n_0$ so that $K_{n_0} \neq 0$.
For all $n \geq n_0$, let $h_n \neq 0 $ be a greatest common divisor of $K_n$, considered as a  subset of $T_n$.
By \eqref{star1}, $h_{n+1} \ | \ h_n x, h_n y$.
Since $x,y$ have no common divisor, we have $h_{n+1} \ |\  h_n$ for all $n \geq n_0$.
This chain of divisors must stabilize, and thus there is $n_1 \geq n_0$ so that $h_{n+1} h_n^{-1} \in \kk$ for $n \geq n_1$. 
Let $h := h_{n_1}$.

By \eqref{star2}, $h \ | \ \mu^m(h) $ for all $m \in \NN$, so $h$ is an eigenvector of $\mu$.  Thus, $h$ is normal in $Q$.  
 Since $h \ | \ f$ for all $f \in K$, we can write $K = hJ$ 
 for some $J \subseteq Q$. Since $h$ is normal,  $J$ is again an ideal  of $Q$. So,  \eqref{star1}, \eqref{star2} apply to $J$.

Since $h\in \kk[v,vw]$ 
and $\kk$ is algebraically closed, we have
\[ h = (vw-a_1 v^2)\cdots (vw-a_n v^2) v^k\]
for some $n, k \in \NN$ and $a_1, \dots, a_n \in \kk$.
Applying Lemma~\ref{lem:normalfact} repeatedly, we obtain that 
 $Q/hQ$ is a noetherian $B$-bimodule.

From the  exact sequence 
\[ 0 \to hQ/hJ \to Q/K \to Q/hQ\to 0,\]
it suffices to prove that $hQ/hJ$ is a noetherian $B$-bimodule. 
We make a geometric argument to do so.

Graded ideals of $T$ correspond to subschemes of the weighted projective plane $\PP(1,1,2)$. Note that $\mu$ acts on $\PP(1,1,2)$ by $\mu(a:b:c) = (a-b:b:c)$.  

Let $Y_n $ be the subset of $\PP(1,1,2)$ defined by the vanishing of the polynomials in $J_n$, considered now as a subset of $T$.
By the definition of $h$, for $n \geq n_1$ the polynomials in $J_n$ have no nontrivial common factor, and so $\dim Y_n \leq 0$.
By \eqref{star1}, \eqref{star2}, 
we have
\[ Y_{n+1} \subseteq Y_n \cap \mu(Y_n)\]
for $n \geq n_1$.
It follows that there exists $n_2 \geq n_1$ so that 
\beq\label{Y}
Y_{n+1} = Y_n = \mu(Y_n)
\eeq
 for $n \geq n_2$. 
 Let $Y :=Y_{n_2}$. Since $\mu$-orbits in $\PP(1,1,2)$ are either infinite or trivial, each point of $Y$ is $\mu$-invariant.  
 Note that $Y$ is the subset of $\PP(1,1,2)$ defined by $J$, considered as an ideal of $T$.

Let $P$ be an associated prime of $J$.  
Since $J$ is graded, $P$ is graded.  
By using the Nullstellensatz, with the fact that $\dim Y \leq 0$, 
 we get that either $P = T_{+}$, or $P$ defines some point $(a:b:c) \in Y$.
In the first case, certainly $y \in P$.
In the second case, $(a:b:c) = \mu(a:b:c) = (a-b:b:c)$ and so  $b = 0$.
Again, $y \in P$.

The radical $\sqrt J$ is the intersection of the associated primes of $J$.  Since $y$ is contained in all associated primes, $y \in \sqrt{J}$.
Thus, there is some $n$ so that $y^n = v^n \in J$. So, $hQ/hJ$ is a factor of $hQ/hv^nQ$.  
Applying Lemma~\ref{lem:normalfact} again, we see that $hQ/hJ$ is a noetherian $B$-bimodule, as desired.
\end{proof}

We now prove Proposition~\ref{iprop3}.  In fact, we show:
\begin{proposition}\label{prop3}
The ring $Q$ is noetherian as a $B$-bimodule.  As a consequence, $B$ satisfies ACC on two-sided ideals.
\end{proposition}

\begin{proof}
Let $\kk'$ be an algebraic closure of $\kk$.
If $Q \otimes_\kk \kk'$ were a noetherian bimodule over $B \otimes_\kk \kk'$, then  $Q$ would be a noetherian $B$-bimodule; this holds as $\kk'$ is faithfully flat over $\kk$  \cite[Exercise~17T]{GW}. So it suffices to prove the result in the case that $\kk$ is algebraically closed.
By standard arguments, it is sufficient to show that $Q$ satisfies ACC on {\em graded} $B$-subbimodules,
 or equivalently, that any nonzero graded $B$-subbimodule of $Q$ is finitely generated.  

Let $K$ be a nonzero graded $B$-subbimodule of $Q$.  
Since $B \supseteq Qp = pQ$ by Lemma~\ref{lem:easy}(c), we have that $K=BKB \supseteq QpKpQ$.  Since $Q$ is noetherian, there is  a finite dimensional graded vector space $V \subseteq K$ with $QpKpQ = QpVpQ$.  

By Proposition~\ref{prop:noeth2}, the $B$-bimodule $Q/QpVpQ$ is noetherian.
The $B$-subbimodule $K/QpVpQ$ of 
\linebreak
$Q/QpVpQ$ is thus finitely generated. So, there is a finite-dimensional vector space $W \subseteq K$ so that $K = BWB + QpVpQ \subseteq BWB + BVB$.
As $V, W \subseteq K$, certainly $K \supseteq BWB + BVB$.
 Thus, $K$ is finitely generated by $ V + W$, as needed.  
\end{proof}


\section{Appendix}

We first give 
a general result from ring theory to which we were not able to find a reference; it is the  converse to \cite[Lemma~2.11]{Rog-notes}. We then finish by presenting Maple and Macaulay2 routines and proofs of computational claims asserted above.

\subsection{A result in ring theory} 
Consider the following setting. Let $T= \kk\ang{t_1,\dots, t_n}$ be the free algebra.  Set $\deg t_i = d_i \in \ZZ_{\geq1}$, and grade $T$ by the induced grading.
Suppose that $\pi: T \to A$ is a surjective  homomorphism of graded algebras, and let $a_i = \pi(t_i)$.  By definition, the $a_i$ generate $A$ as an algebra.   Let $J = \ker \pi$. Consider the map 
$$\xymatrix{ \alpha:  A[-d_1] \oplus \dots \oplus A[-d_n] \ar[rr]^(0.7){(a_1, \dots, a_n)} && A}$$
that sends $(r_1, \dots, r_n) \mapsto \sum_{i=1}^n a_i r_i$.  Note $\alpha$ is a homomorphism of graded right $A$-modules, and set $K = \ker \alpha$.
Let $b^1, \dots, b^m$ be homogeneous elements of $K$, where $b^j = (b^j_1, \dots, b^j_n) \in A[-d_1] \oplus \dots \oplus A[-d_n]$.
For all $1 \leq i \leq n, 1\leq j \leq m$, choose homogenous  elements $\wt{b}^j_i \in T$ so that $\pi(\wt{b}^j_i ) =b^j_i$.
Let $q_j = \sum_{i=1}^n t_i \wt{b}^j_i$.  (Note that the $q_i$ are homogeneous; in fact, $\deg q_j = \deg b^j$.)

\begin{proposition}\label{propappendix}
Retain the notation above.
If $\{b^1, \dots, b^m\}$ generate $K$ as a right $A$-module, then \linebreak $\{q_1, \dots, q_m\}$ generate $J$ as an ideal of $T$.
\end{proposition}

\begin{proof}
Let $J'$ be the ideal of $T$ generated by $q_1, \dots, q_m$.  Since
$ \pi(q_j) = \sum_i \pi(t_i) \pi(\wt{b}^j_i) = \sum_i a_i b^j_i = \alpha(b^j)=0,$ we get that
$J' \subseteq J$.

We prove by induction that $J'_k = J_k$ for all $k \in \NN$.  Certainly $J'_0 = J_0 = 0$.  
Assume that we have shown that $J'_{<k} = J_{<k}$, and let $h \in J_k$.
Because $T$ is generated by $t_1, \dots, t_n$, there are homogeneous  elements $f_1, \dots, f_n \in T$, with $\deg f_i =k-d_i$, so that $h  = \sum_i t_i f_i$.
Then,
\[ 0 = \pi(h) = \sum_{i=1}^n a_i \pi(f_i) = \alpha(\pi(f_1), \dots, \pi(f_n)).\]
Since the $b^j$ generate $K = \ker \alpha$, there are homogeneous elements $r_1, \dots, r_m \in A$ 
with 
$(\pi(f_1), \dots, \pi(f_n)) = \sum_{j=1}^m b^j r_j.$
Let $\wt{r}_1, \dots, \wt{r}_m$ be homogeneous lifts of $r_1, \dots, r_m$.
Then for each $i$ we have
\[ \pi(f_i)  = \sum_j b^j_i r_j = \sum_j \pi(\wt{b}^j_i \wt{r}_j).\]
So,  $f_i - \sum_j \wt{b}^j_i \wt{r}_j \in J = \ker \pi$.  Since $\deg f_i = k - d_i < k$, each $f_i - \sum_j \wt{b}^j_i \wt{r}_j \in J'$.
Thus $J'$ contains
\[
\sum_i t_i f_i - \sum_i t_i \left(\sum_j \wt{b}^j_i \wt{r}_j\right) = h - \sum_j \left(\sum_i t_i \wt{b}^j_i\right) \wt{r}_j = h - \sum_j q_j \wt{r}_j.\]
As $\sum_i t_i \wt{b}^j_i = q_j \in J'$ by definition, we have $\sum_j q_j \wt{r}_j \in J'$. Therefore,   $h \in J'_k$.
\end{proof}

\smallskip

\subsection{Proof of assertions: Maple routines} We begin with the following Maple routine.

\begin{routine} \label{rout:kernel}
A Maple routine to compute the kernel of $\lambda_a$ at a specific degree $n$ is presented as follows.
\end{routine}

\noindent Recall from Lemma~\ref{lem:present}(b) that a $\kk$-vector space basis of $U(W_+)_n$ is given by partitions of $n$. Moreover, we employ Lemma~\ref{Zhang} and~\eqref{multiplyS} to input a function $f(i,j) = \lambda_a(e_i)^{\mu^j},$ considered as an element of $\kk[x,y]$.
\medskip

{\footnotesize
\begin{verbatim}
with(combinat,partition):   with(LinearAlgebra):
# Choose value of n
n:=1;   
N:=partition(n):   f:=(i,j)->((x-j*y)-(i-1)*a*y)*y^(i-1):
\end{verbatim}
}
\medskip

\noindent Given a partition $d:=(n_1, \dots, n_k)$ of $n$, we create a list of double indexed entries 
$m = (m[i_1,j_1], \dots, m[i_k,j_k])$. Here, $i_\ell = n_\ell$, and $j_1=0$ with $j_\ell = j_{\ell-1} + n_{\ell-1}$ for $\ell \geq 2$. Then, $\lambda_a(e_{n_1} \cdots e_{n_k}) = m[i_1,j_1]\cdots m[i_k,j_k]$, denoted by $P$. (Here, $P$ in list form, which we put in matrix form later for multiplication. The $k$-loop enables us to form the product of elements $m[i_*, j_*]$.)
\medskip

{\footnotesize
\begin{verbatim}
P:=[]:
for d from 1 to nops(N) do         M:=[]:                    j[1]:=0:
for l from 1 to nops(N[d]) do      j[l+1]:=j[l]+ N[d][l]:    M:=[op(M),f(N[d][l],j[l])]:    S[0]:=1:
for k from 1 to nops(M) do         S[k]:=S[k-1]*M[k]:
end do:   end do:
P:=[op(P),expand(S[nops(M)])]:
end do:
\end{verbatim}
}
\medskip

\noindent Next, we define an arbitrary element of $\lambda_a(U(W_+)_n)$, namely $p := \sum_{i=1}^k b_i \lambda_a(e_{n_i})$.
\medskip

{\footnotesize
\begin{verbatim}
B:=[];
for i from 1 to nops(N) do         B:=[op(B),b[i]]:          end do:
Bvec:=convert(B,Matrix):           Pvec:=convert(P,Matrix):  q:=Multiply(Bvec,Transpose(Pvec)):
p:=expand(q[1][1]):
\end{verbatim}
}

\medskip

\noindent Then, we set the coefficients of $p$ equal to 0 and solve for the $b_i$. We rule out the case when $a =0,1$.
\medskip

{\footnotesize
\begin{verbatim}
Coeffs:=[coeffs(collect(p,[x,y], 'distributed'),[x,y])]:
solve([op(Coeffs),a<>0,a<>1]);
\end{verbatim}
}
\medskip

\noindent Note that the number of free $b_i$ equals the $\kk$-vector space dimension of $(\ker \lambda_a)_n$.
\medskip

We continue by verifying the claim from the proof of Lemma~\ref{lem:lambdaa}.

\begin{claim} \label{aclaim1} Retain the notation from Section~\ref{LAMBDA}, especially in Lemma~\ref{lem:lambdaa}. We have that $J_5A(a)_2 \not \subseteq J_6 A(a)_1$.
\end{claim}

\begin{proof}
Nonzero elements in $J_5$ arise as elements of $(u-av)vA(a)_3$ that are divisible by $u$ on the left. We obtain that 
\[
\begin{array}{rl}
(u-av)vA(a)_3&=\kk[(uv-av^2)(u^3)] \oplus \kk[(uv-av^2)(u(u-av)v)] \oplus \kk[(uv-av^2)((u-2av)v^2)]
\\
&=\kk[r_1] \oplus ~\kk[r_2] \oplus ~\kk[r_3]
\end{array}
\]
where 
\[
\begin{array}{l}
r_1 := u^4v - (3+a)u^3v^2 +(6+6a)u^2v^3 - (6+18a)uv^4 +24av^5,\\
r_2 := u^3v^2 - (2+2a) u^2v^3 + (2+5a+a^2)uv^4 -(6a+2a^2)v^5,\\
r_3 := u^2v^3 - (1+3a)uv^4 + (2a+2a^2)v^5.
\end{array}
\]
We see this as $v^k u = uv^k - kv^{k+1}$ for all $k \geq 1$, $vu^2  = u^2v - 2uv^2 + 2v^3$, $v^2 u^2 = u^2v^2 - 4uv^3 + 6v^4$, $vu^3 = u^3v - 3u^2 v^2 + 6uv^3 - 6v^4$, and $v^2u^3 = u^3v^2 - 6u^2 v^3 + 18u v^4 - 24 v^5$ in $R$.
Eliminating the $v^5$ term of $r_1, r_2, r_3$, we get that $J_5$ is generated by 
\[
\begin{array}{rl}
s_1 &:= (3+a)r_1 + 12r_2,\\
s_2&:=(1+a)r_1 -12r_3,\\
s_3&:=(1+a)r_2 + (3+a)r_3.
\end{array}
\]

By way of contradiction,  suppose that $J_5A(a)_2\subseteq J_6 A(a)_1$. Recall that $J \subseteq L$, where $L:=uR \cap (u-av)vR$. Further, $J_6=L_6$, and $L=rR$ for 
$$r=u(uv+(1-a)v^2) = (uv-av^2)(u+2v).$$ 
So, $s_i = r(c_{i1} u^2 + c_{i2}uv +c_{i3}v^2) \in J_5 \subseteq rR_2$, for some $c_{ij} \in \kk$. We produce these coefficients $c_{ij}$ below.

{\footnotesize
\begin{verbatim}
r1:=x*(x-y)*(x-2*y)*(x-3*y)*y-(3+a)*x*(x-y)*(x-2*y)*y^2+(6+6*a)*x*(x-y)*y^3-(6+18*a)*x*y^4+24*a*y^5:
r2:=x*(x-y)*(x-2*y)*y^2-(2+2*a)*x*(x-y)*y^3+(2+5*a+a^2)*x*y^4-(6*a+2*a^2)*y^5:
r3:=x*(x-y)*y^3-(1+3*a)*x*y^4+(2*a+2*a^2)*y^5:
s1:=(3+a)*r1+12*r2:                s2:=(1+a)*r1-12*r3:              s3:=(1+a)*r2+(3+a)*r3:
r:=x*((x-y)*y+(1-a)*y^2):
eq1:=s1 - r*(c11*(x-3*y)*(x-4*y)+c12*(x-3*y)*y+c13*y^2):
eq2:=s2 - r*(c21*(x-3*y)*(x-4*y)+c22*(x-3*y)*y+c23*y^2):
eq3:=s3 - r*(c31*(x-3*y)*(x-4*y)+c32*(x-3*y)*y+c33*y^2):
Coeffs1:=[coeffs(collect(eq1,[x,y], 'distributed'),[x,y])]:
Coeffs2:=[coeffs(collect(eq2,[x,y], 'distributed'),[x,y])]:
Coeffs3:=[coeffs(collect(eq3,[x,y], 'distributed'),[x,y])]:
solve(Coeffs1);                   solve(Coeffs2);                   solve(Coeffs3);
>             {a = a, c11 = 3 + a, c12 = 6 - 2 a, c13 = -4 a}
>             {a = a, c21 = 1 + a, c22 = -2 - 2 a, c23 = -4 + 8 a}
                                                             2
>             {a = a, c31 = 0, c32 = 1 + a, c33 = 1 - 2 a - a }
\end{verbatim}
}
\noindent Therefore,
\[
\begin{array}{rl}
 s_1&=r((3+a)u^2+(6-2a)uv-4av^2),\\
 s_2&=r((1+a)u^2-(2+2a)uv-(4-8a)v^2),\\ 
 s_3&=r((1+a)uv+(1-2a-a^2)v^2).
 \end{array}
 \]

Now by assumption, for $i = 1,2,3$ we have $s_i (u-av) v= w_i u $ for some $w_i \in J_6$.
Take an arbitrary element of $J_6 = L_6 = rR_3$, namely $r(d_{i1}u^3+d_{i2}u^2v+d_{i3}uv^2+d_{i4}v^3)$ for $d_{ij} \in \kk$. Then, for some $\alpha_i \in \kk$,
\beq\label{tabitha}
p_i:=s_i (u-av)v= \alpha_i r(d_{i1}u^4+d_{i2}u^2vu+d_{i3}uv^2u+d_{i4}v^3u).
\eeq
Continuing with the code we enter:

{\footnotesize
\begin{verbatim}
s1:=r*((3+a)*(x-3*y)*(x-4*y)+(6-2*a)*(x-3*y)*y-4*a*y^2):
s2:=r*((1+a)*(x-3*y)*(x-4*y)-(2+2*a)*(x-3*y)*y-(4-8*a)*y^2):
s3:=r*((1+a)*(x-3*y)*y+(1-2*a-a^2)*y^2):
p1:=s1*(x-(5+a)*y)*y:              p2:=s2*(x-(5+a)*y)*y:         p3:=s3*(x-(5+a)*y)*y:
Eq1:=p1 - alpha1*r*(d11*(x-3*y)*(x-4*y)*(x-5*y)*(x-6*y) + d12*(x-3*y)*(x-4*y)*y*(x-6*y)
                   +d13*(x-3*y)*y^2*(x-6*y) + d14*y^3*(x-6*y)):
Eq2:=p2 - alpha2*r*(d21*(x-3*y)*(x-4*y)*(x-5*y)*(x-6*y) + d22*(x-3*y)*(x-4*y)*y*(x-6*y)
                   +d23*(x-3*y)*y^2*(x-6*y) + d24*y^3*(x-6*y)):
Eq3:=p3 - alpha3*r*(d31*(x-3*y)*(x-4*y)*(x-5*y)*(x-6*y) + d32*(x-3*y)*(x-4*y)*y*(x-6*y)
                   +d33*(x-3*y)*y^2*(x-6*y) + d34*y^3*(x-6*y)):
CCoeffs1:=[coeffs(collect(Eq1,[x,y], 'distributed'),[x,y])]:
CCoeffs2:=[coeffs(collect(Eq2,[x,y], 'distributed'),[x,y])]:
CCoeffs3:=[coeffs(collect(Eq3,[x,y], 'distributed'),[x,y])]:
L1:=solve(CCoeffs1):               L2:=solve(CCoeffs2):          L3:=solve(CCoeffs3):
for i from 1 to nops([L1]) do      print(L1[i][1]);              end do;
>              a = 9,    a = 1
for i from 1 to nops([L2]) do      print(L2[i][1]);              end do;
>              a = 1,    a = 1/2
for i from 1 to nops([L3]) do      print(L3[i][1]);              end do;
                                                  2
>              a = 1,    a = RootOf(-2 - 3 _Z + _Z ) - 1
\end{verbatim}
}
\noindent 
So in order for \eqref{tabitha} to hold for $i = 1,2,3$, we must have $a=1$. This yields a contradiction as desired.
\end{proof}

We now verify the claim from the proof of Proposition~\ref{prop:(d)fora}.

\begin{claim} \label{aclaim2} Retain the notation from Section~\ref{LAMBDA}, especially in Proposition~\ref{prop:(d)fora}. We have that  $h_2$, $h_3$, $e_1h_1$, $h_1e_1$ are $\kk$-linearly independent  and that 
$$h_4=2a(2a+1)h_2-h_3-(6+4a)e_1h_1+(2+4a)h_1e_1, \quad
h_5=4a^2h_2 -h_3-(4+4a)e_1h_1 +(4a)h_1e_1.$$
\end{claim}

\begin{proof}
This is established simply by considering the following linear combination
$$c_1h_2+c_2h_3+c_3h_4+c_4h_5+c_5e_1h_1+c_6h_1e_1,$$
setting the coefficients of the basis elements of $U(W_+)_6$ equal to 0, and solving for $c_1, \dots, c_6$. By Lemma~\ref{lem:present}(a), the basis elements of $U(W_+)_6$ are
$$e_1^6,~ e_1^4e_2, ~e_1^2e_2^2,~e_2^3,~e_1^3e_3,~e_1e_2e_3,~e_3^2,~e_1^2e_4,~e_2e_4,~e_1e_5,~e_6.$$
So, we establish the claim via the following Maple routine:
{\footnotesize
\begin{verbatim}
with(LinearAlgebra):
M:=Matrix([
[0,  0,  0,  0,  0,      0,             3,     0,   -4,               1,                     2],
[0,  0, -4, -4,  4,      0, 20*a^2+14*a-7,     0,    0,  -16*a^2-18*a-5,  16*a^3+36*a^2+16*a-2],
[0,  0,  0,  4,  0,     -4,         7-4*a,     0,    0,           4*a+1,         -4*a^2- 4*a+2],
[0,  0,  0,  4,  0,      0,        7-14*a,    -4,    0,          14*a+5,        -12*a^2-16*a+2],
[0,  0,  1,  0, -1,   -2*a,             0, 2*a+1,    0,          -a^2-a,                     0],
[0,  0,  1,  0, -1, -2*a-2,           2*a, 2*a+3,  4*a,      -a^2-7*a-2,             4*a^2+4*a]
]);
P:=Matrix([
[c1,  0,  0,  0,  0,  0],
[ 0, c2,  0,  0,  0,  0],
[ 0,  0, c3,  0,  0,  0],
[ 0,  0,  0, c4,  0,  0],
[ 0,  0,  0,  0, c5,  0],
[ 0,  0,  0,  0,  0, c6]
]);
B:=Multiply(P,M);  
for i from 1 to 11 do          L[i]:=B[1,i]+B[2,i]+B[3,i]+B[4,i]+B[5,i]+B[6,i]:          end do:
V:=solve([L[1],L[2],L[3],L[4],L[5],L[6],L[7],L[8],L[9],L[10],L[11]],[c1,c2,c3,c4,c5,c6]);
>[[c1 = -2 (c3 + 2 c3 a + 2 c4 a) a,    c2 = c3 + c4,    c3 = c3,    c4 = c4,
   c5 = 6 c3 + 4 c4 + 4 c3 a + 4 c4 a,  c6 = -2 c3 - 4 c3 a - 4 c4 a]]
eval(V,[c3=1,c4=0]);
>  [[c1 = -2 (2 a + 1) a, c2 = 1, 1 = 1, 0 = 0, c5 = 6 + 4 a, c6 = -2 - 4 a]]
eval(V,[c3=0,c4=1]);
>                    2
          [[c1 = -4 a , c2 = 1, 0 = 0, 1 = 1, c5 = 4 + 4 a, c6 = -4 a]]
\end{verbatim}
}
\vspace{-.3in}

\end{proof}

\medskip

Now verify the claims from the proof of Lemma~\ref{finalclaim}.

\begin{claim} \label{claim:final} Retain the notation from   Lemma~\ref{finalclaim}.  We have the following statements.
\begin{enumerate}
\item $b_5 Q + b_6 Q + b_7 Q \ni~~ x(xy-yz)(xyz+y^2z) = (uv-vw)(u+2v)(u+4v)vw.$
\item $(M' \cap \ker \eta) \supseteq hQ_i$ for $i \leq 2$, where $$h = (uv-vw)(u+2v)(v^3w-v^2w^2) = (xy-yz)x(y^3z-y^2z^2).$$
\end{enumerate}
\end{claim}

\begin{proof}
(a) We see that $-\frac{1}{6}b_5 u+ b_5 v + \frac{1}{6} b_6 = (uv-vw)(u+2v)(u+4v)vw $ by using Lemma~\ref{Zhang} and \eqref{multiplyS}:

{\footnotesize
\begin{verbatim}
b5:=(x*y-y*z)*((x-2*y)*(x-3*y)*(x-4*y)-6*((x-2*y)*y-y*z)*(x-4*y)+12*(x-2*y)*((x-3*y)*y-y*z)):
b6:=(x*y-y*z)*(-48*((x-2*y)*y-3*y*z)*y^2-36*(x-2*y)*((x-3*y)*y-2*y*z)*y+(x-2*y)*(x-3*y)*(x-4*y)*(x-5*y)):
r:=x*(x*y-y*z)*(x*y*z+y^2*z):
p:=c1*b5*(x-5*y)+c2*b5*y+c3*b6 - r:
Coeffs:=[coeffs(collect(p,[x,y,z], 'distributed'),[x,y,z])]:
solve(Coeffs);
>                       {c1 = -1/6, c2 = 1, c3 = 1/6}
\end{verbatim}
}

(b) It is easy to see that $\eta(h) =0$, so it suffices to show that $hQ_0$, $hQ_1$, $hQ_2$ are in $M':=b_5B+b_6B+b_7B$. Recall that $Q$ is the subalgebra of $S$ generated by $u,v,vw$, and $B$ is the subalgebra of $S$ generated by $u, uv-vw$. Since deg($h$) =7, 
\[
\begin{array}{rl}
hQ_0 &= \{c_1 h ~|~ c_1 \in \kk\},\\
hQ_1 &= \{c_2 h u + c_3 h v ~|~ c_i \in \kk\},\\
hQ_2 &= \{c_4 h u^2 + c_5 h uv + c_6 h v^2 + c_7 h vw ~|~ c_i \in \kk\},
\end{array}
\]
and moreover,
\[
\begin{array}{rl}
M'_7 &= \{d_1 b_5 u^2 + d_2 b_5(uv-vw) + d_3 b_6 u + d_4 b_7 ~|~ d_i \in \kk\},\\
M'_8 &= \{d_5 b_5 u^3 + d_6 b_5u(uv-vw) + d_7 b_5(uv-vw)u + d_8 b_6 u^2 + d_9 b_6 (uv-vw)  + d_{10} b_7u ~|~ d_i \in \kk\},\\
M'_9 &= \{d_{11} b_5 u^4 + d_{12} b_5u^2(uv-vw) + d_{13} b_5u(uv-vw)u  + d_{14} b_5(uv-vw)u^2 + d_{15} b_5(uv-vw)^2\\
&\hspace{.7in}+ d_{16} b_6 u^3 + d_{17} b_6 u(uv-vw)  + d_{18} b_6 (uv-vw)u  + d_{19} b_7u^2 + d_{20} b_7(uv-vw) ~|~ d_i \in \kk\},\\
\end{array}
\]

Continuing with the code in part (a), we enter:
{\footnotesize
\begin{verbatim}
b7:=(x*y-y*z)*((x-2*y)*(x-3*y)*(x-4*y)*(x-5*y)*(x-6*y)-40*(((x-2*y)*y-y*z)*((x-4*y)*y-y*z)*(x-6*y)
               -3*((x-2*y)*y-y*z)*(x-4*y)*((x-5*y)*y-y*z)+3*(x-2*y)*((x-3*y)*y-y*z)*((x-5*y)*y-y*z))):
h:=(x*y-y*z)*x*(y^3*z-y^2*z^2):
hQ0:=c1*h:
hQ1:=c2*h*(x-7*y)+c3*h*y:
hQ2:=c4*h*(x-7*y)*(x-8*y)+c5*h*(x-7*y)*y+c6*h*y^2+c7*h*y*z:
m7:=d1*b5*(x-5*y)*(x-6*y)+d2*b5*((x-5*y)*y-y*z)+d3*b6*(x-6*y)+d4*b7: 
m8:=d5*b5*(x-5*y)*(x-6*y)*(x-7*y)+d6*b5*(x-5*y)*((x-6*y)*y-y*z)+d7*b5*((x-5*y)*y-y*z)*(x-7*y)
                                 +d8*b6*(x-6*y)*(x-7*y)+d9*b6*((x-6*y)*y-y*z)+d10*b7*(x-7*y):
m9:=d11*b5*(x-5*y)*(x-6*y)*(x-7*y)*(x-8*y)+d12*b5*(x-5*y)*(x-6*y)*((x-7*y)*y-y*z)
                                 +d13*b5*(x-5*y)*((x-6*y)*y-y*z)*(x-8*y)+d14*b5*((x-5*y)*y-y*z)*(x-7*y)*(x-8*y)
                                 +d15*b5*((x-5*y)*y-y*z)*((x-7*y)*y-y*z)+d16*b6*(x-6*y)*(x-7*y)*(x-8*y)
                                 +d17*b6*(x-6*y)*((x-7*y)*y-y*z)+d18*b6*((x-6*y)*y-y*z)*(x-8*y)
                                 +d19*b7*(x-7*y)*(x-8*y)+d20*b7*((x-7*y)*y-y*z):
p7:=m7 - hQ0:            p8:=m8 - hQ1:            p9:=m9 - hQ2:
Coeffs7:=[coeffs(collect(p7,[x,y,z], 'distributed'),[x,y,z])]:
Coeffs8:=[coeffs(collect(p8,[x,y,z], 'distributed'),[x,y,z])]:
Coeffs9:=[coeffs(collect(p9,[x,y,z], 'distributed'),[x,y,z])]:
solve(Coeffs7,[d1,d2,d3,d4]);
                                         c1         c1           c1         c1
>                              [[d1 = - ----, d2 = ----, d3 = - ----, d4 = ----]]
                                         24         4            48         16
solve(Coeffs8,[d5,d6,d7,d8,d9,d10]);
                c2     c3         c3         c2     c3           c2    c3         c3          c2     c3
>     [[d5 = - ---- - ----, d6 = ----, d7 = ---- + ----, d8 = - ---- + ---, d9 = ----, d10 = ---- + ----]]
                24     48         24         4      16           48    192        48          16     64
solve(Coeffs9,[d11,d12,d13,d14,d15,d16,d17,d18,d19,d20]);
                  c4    c6     c5     c7                              9 c4    c6    25 c5   11 c7
>[[d11 = 8 d16 + ---- + --- - ---- - ----,  [...],   d20 = -108 d16 - ---- - ---- + ----- + -----]]
                  8     144    18     18                               4      24     48      24
 \end{verbatim}
}
\noindent Thus, all arbitrary elements of $hQ_0$, $hQ_1$, $hQ_2$ are contained, respectively, in $M'_7$, $M'_8$, $M'_9$, as desired.
\end{proof}

\medskip

\subsection{Proof of assertions: Macaulay2 routines}\label{M2}

The following Macaulay2 code verifies Lemma~\ref{lem:psi}(b) and~\eqref{two}; see lines o7-o10 and line o13, respectively.

{\footnotesize
\begin{verbatim}
Macaulay2, version 1.4
with packages: ConwayPolynomials, Elimination, IntegralClosure, LLLBases, PrimaryDecomposition,
ReesAlgebra, TangentCone
i1 : ringX=QQ[w,x,y,z]/ideal(x*z-y^2);
i2 : taustar=map(ringX,ringX,{w-2*x+2*z,z,-y-2*z,x+4*y+4*z});
i3 : ringP1a=QQ[x,y,a];
i4 : mustar=map(ringP1a, ringP1a, {x-y,y,a});
i5 : psistar=map(ringP1a, ringX, {2*x^2-4*x*y-6*a*y^2,x^2-2*x*y+y^2,-x^2+3*x*y-2*y^2,x^2-4*x*y+4*y^2});
i6 : use ringX;
i7 : mustar(psistar(w))==psistar(taustar(w))        o7 = true
i8 : mustar(psistar(x))==psistar(taustar(x))        o8 = true
i9 : mustar(psistar(y))==psistar(taustar(y))        o9 = true
i10 : mustar(psistar(z))==psistar(taustar(z))       o10 = true
i11 : num=w+12*x+22*y+8*z;
i12 : den=12*x+6*y;
                                                             2
                                                          - y a + x*y
i13 : psistar(num)/psistar(den)                     o13 = -----------          o13 : frac(ringP1a)
                                                             2
                                                            x  - x*y
\end{verbatim}
}

\section*{Acknowledgements}
C. Walton was supported by the US National Science Foundation grant \#1550306.      
  We thank MIT and the NSF for supporting a visit by the first author to MIT in February 2015, in which much of the work on this paper was done.
We thank  Jacques Alev, Jason Bell, and Lance Small for helpful discussions.
In addition, we thank the referee for pointing us to the reference \cite{CM07} and for prompting  Remarks~\ref{rem:one} and~\ref{rem:three}.

\bibliography{Witt_biblio}

\begin{thebibliography}{GW04}

\bibitem[AS95]{AS}
M.~Artin and J.~T. Stafford.
\newblock Noncommutative graded domains with quadratic growth.
\newblock {\em Invent. Math.}, 122(2):231--276, 1995.

\bibitem[CM07]{CM07}
C.~H. Conley and C.~Martin.
\newblock Annihilators of tensor density modules.
\newblock {\em J. Algebra}, 312(1):495--526, 2007.

\bibitem[GW04]{GW}
K.~R. Goodearl and R.~B. Warfield, Jr.
\newblock {\em An introduction to noncommutative {N}oetherian rings}, volume~61
  of {\em London Mathematical Society Student Texts}.
\newblock Cambridge University Press, Cambridge, second edition, 2004.

\bibitem[Har77]{Hartshorne}
R.~Hartshorne.
\newblock {\em Algebraic Geometry}, volume~52 of {\em Graduate Texts in
  Mathematics}.
\newblock Springer-Verlag, New York, 1977.

\bibitem[Har92]{Harris}
J.~Harris.
\newblock {\em Algebraic geometry}, volume 133 of {\em Graduate Texts in
  Mathematics}.
\newblock Springer-Verlag, New York, 1992.
\newblock A first course.

\bibitem[Rog]{Rog-notes}
D.~Rogalski.
\newblock An {I}ntroduction to {N}oncommutative {P}rojective {A}lgebraic
  {G}eometry (lecture notes).
\newblock \url{http://arxiv.org/abs/1403.3065}.

\bibitem[Sie11]{S-surfclass}
S.~J. Sierra.
\newblock Classifying birationally commutative projective surfaces.
\newblock {\em Proceedings of the LMS}, 103:139--196, 2011.

\bibitem[Sta85]{Stafford85}
J.~T. Stafford.
\newblock On the ideals of a {N}oetherian ring.
\newblock {\em Trans. Amer. Math. Soc.}, 289(1):381--392, 1985.

\bibitem[SW14]{SieWal:Witt}
S.~J. Sierra and C.~Walton.
\newblock The universal enveloping algebra of the {W}itt algebra is not
  noetherian.
\newblock {\em Adv. Math}, 262:239--260, 2014.

\bibitem[SZ94]{SZ}
J.~T. Stafford and J.~J. Zhang.
\newblock Examples in non-commutative projective geometry.
\newblock {\em Math. Proc. Cambridge Philos. Soc.}, 116(3):415--433, 1994.

\bibitem[Zha96]{Zhangtwist}
J.~J. Zhang.
\newblock Twisted graded algebras and equivalences of graded categories.
\newblock {\em Proc. London Math. Soc. (3)}, 72(2):281--311, 1996.

\end{thebibliography}

\end{document}